\documentclass[12pt, reqno]{amsart}
\usepackage{amsmath, amssymb, amsthm, graphicx, hyperref}
\usepackage[margin=1in]{geometry}

\theoremstyle{definition}
\newtheorem{theorem}{Theorem}[section]
\newtheorem{lemma}[theorem]{Lemma}
\newtheorem{definition}[theorem]{Definition}
\newtheorem{remark}[theorem]{Remark}
\numberwithin{equation}{section}

\newtheorem{proposition}[theorem]{Proposition}

\newtheorem{example}[theorem]{Example}
\newtheorem{corollary}[theorem]{Corollary} 

\usepackage{subcaption}
\usepackage{amsfonts}
\usepackage{ytableau}

\newcommand{\Comp}{\mathrm{Comp}}
\newcommand{\Set}{\mathrm{Set}}

\newcommand{\Row}{\mathrm{Row}}
\newcommand{\Des}{\mathrm{Des}}
\newcommand{\content}{\mathrm{content}}
\newcommand{\Span}{\mathrm{span}}

\newcommand{\flip}{\mathrm{flip}}
\newcommand{\Id}{\mathrm{Id}}

\newcommand{\SIT}{\mathrm{SIT}}
\newcommand{\SRIT}{\mathrm{SRIT}}
\newcommand{\SSIT}{\mathrm{SSIT}}
\newcommand{\RSSIT}{\mathrm{RSSIT}}
\newcommand{\SSRIT}{\mathrm{SSRIT}}
\newcommand{\RSSRIT}{\mathrm{RSSRIT}}

\newcommand{\Sym}{\ensuremath{\operatorname{Sym}}}
\newcommand{\QSym}{\ensuremath{\operatorname{QSym}}}

\newcommand{\BdI}{\rotatebox[origin=c]{180}{$\mathfrak{S}$}}
\newcommand{\dI}{\mathfrak{S}}

\newcommand{\svw}[1]{\textcolor{black}{#1}}
\newcommand{\svwrev}[1]{\textcolor{black}{#1}}
\newcommand{\ME}[1]{\textcolor{black}{#1}}

\begin{document}
\title{Equality of Dual Immaculate Functions Under Automorphisms}
\author{Maria Esipova}
\address{Department of Mathematics and Statistics, McGill University, Montréal, Canada} 
\email{maria.esipova@mail.mcgill.ca}
\author{Stephanie van Willigenburg}
\address{Department of Mathematics, University of British Columbia,  Vancouver, Canada}
\email{steph@math.ubc}

\thanks{All authors were supported  in part by the Natural Sciences and Engineering Research Council of Canada.}
\subjclass[2020]{05A15, 05E05, 16T30}
\keywords{automorphisms, dual immaculate function, quasisymmetric function, tableaux combinatorics}
\date{\today}

\begin{abstract}
The dual immaculate functions are an example of a Schur-like basis in the algebra of quasisymmetric \svw{functions.} We classify when the image of a dual immaculate function under one of the involutions $\rho,\psi,\omega$ is equal to a dual immaculate function. As well, the leading term of the transition matrix is identified, and sufficient and necessary conditions for the existence of immaculate tableaux \svw{are determined.} As a consequence, new maps and canonical tableaux associated with compositions are discovered. 
\end{abstract}

\maketitle
\tableofcontents

\section{Introduction} \label{ch:intro}
The algebra of symmetric functions ($\Sym$) is ubiquitous in the area of algebraic combinatorics and provides connections to areas such as representation theory, algebra, statistical mechanics, discrete geometry, algebraic geometry, and of course \svw{combinatorics.} Many of these connections are through the Schur basis of \svw{$\Sym$.} 

The algebra of quasisymmetric functions ($\QSym$) generalizes \ME{the algebra of} symmetric functions, containing $\Sym$ as a subalgebra \cite{gessel}. The first Schur-like basis of $\QSym$ was discovered in 2011 \cite{haglundQuasisymmetricSchurFunctions2011}, followed by the dual immaculate basis in 2014 \cite{bergLiftSchurHallLittlewood2014}, and the extended basis in 2018 \cite{assafKohnertPolynomials2018}.
All Schur-like bases can be defined as generating functions of tableaux of composition shapes, generalizing the Young tableaux which define the Schur functions.

In $\Sym$, the Schur basis is invariant under the classical involution $\omega$. However, the dual immaculate basis is not invariant under the maps $\rho,\psi,\omega$ on $\QSym$ which collectively generalize $\omega$ on $\Sym$. On the contrary, the dual immaculate basis, and the images of the dual immaculate basis under $\rho,\psi,\omega$ form a system of distinct Schur-like bases, \svw{each of which has generated their own avenue of research, as we will see.} This is an important distinction between the Schur basis, and its quasisymmetric generalizations.  

The system formed by a basis and its image under $\rho$ is also referred to as the Young and reverse dichotomy. For results on the Young and reverse dichotomy for dual immaculate functions see \ME{\cite{masonLiftingDualImmaculate2021}}  for quasisymmetric Schur functions \cite{luotoIntroductionQuasisymmetricSchur2013a}, for extended functions \cite{daughertyExtendedSchurFunctions2024}, and for polynomial bases \ME{\cite{masonYoungReverseDichotomy2021}.} For Schur functions, the Young and reverse dichotomy is observed by the equivalent definitions of Schur functions as generating functions of Young tableaux as well as reverse tableaux \cite{luotoIntroductionQuasisymmetricSchur2013a}.

For results on the image under the $\psi$ map for dual immaculate functions see \cite{nieseRowstrictDualImmaculate2023}, and for extended functions \cite{daughertyExtendedSchurFunctions2024}. Finally, for results on the image under the $\omega$ map for dual immaculate functions see \cite{daughertySchurlikeBasesTheir2024}, for extended functions \cite{daughertyExtendedSchurFunctions2024}, and for quasisymmetric Schur functions \cite{masonRowstrictQuasisymmetricSchur2011}.

A central area of study in symmetric and quasisymmetric function theory is the transition matrix between different bases. In particular, we wish to understand the expansion of a variant of a dual immaculate function into the dual immaculate basis. 
In this work, we classify when this expansion is `simplest', i.e. consists of exactly one term. In other words, when a variant of a dual immaculate function is equal to a dual immaculate function. A similar result is obtained in \svw{\cite[Theorem 13]{masonYoungReverseDichotomy2021}} where Mason and Searles classify when a quasisymmetric Schur polynomial is equal to a Young quasisymmetric Schur polynomial. As a result of our investigation, we naturally identify three interesting maps on compositions, $\alpha \mapsto \beta^f_\alpha$, as well as four canonical tableaux associated to a composition, $T^{f,M}_\alpha$, which are \svw{combinatorially} interesting in their own right.

In Section \ref{ch:comp_bg} we introduce the relevant background on compositions. In Section \ref{ch:sym_bg} we give the necessary background for the algebras $\Sym$ and $\QSym$, as well as the Schur basis. Section \ref{ch:dI_bg} gives an overview of the dual immaculate basis and its variants. In Section \ref{ch:comps_and_complement}, we prove some results for complements of compositions. In Section \ref{ch:leading_terms}, we identify the leading terms in the dual immaculate \svw{function} \svw{expansions} of the images of the dual immaculate functions under \svw{$\rho,\psi,\omega$.}
In Section \ref{ch:dI_in_F_classification} we classify when a dual immaculate function is equal to a fundamental \svw{quasisymmetric} function. Section \ref{ch:nec_suff_conds} contains necessary and sufficient conditions for the existence of immaculate tableaux with specific shape and descent composition. Finally, Section \ref{ch:dI_images} contains the main result, classifying when a variant of a dual immaculate function is equal to a dual immaculate function.

\section{Background} \label{ch:bg_chapter}

\subsection{Compositions}\label{ch:comp_bg}
Throughout this work, $n$ denotes a positive integer. Denote $\{ 1,2,\dots,n \}$ by $[n]$. For a set of integers $S$ and $a \in \mathbb{Z}$, let $S+a = \{ s+ a \mid s \in S \}$ and similarly $S-a = \{ s- a \mid s \in S \}$.\\

    A \textit{composition} $\alpha$ is a finite ordered list of positive integers, $\alpha=(\alpha_{1},\svw{\ldots},\alpha_{\ell})$. When $\alpha_{j+1} = \svw{\cdots} = \alpha_{j+m} = i$ we often abbreviate this sublist to $i^m$. We say $\alpha$ is a \textit{composition of size} $n$ if $\sum_{i=1}^\ell \alpha_{i}=n$, denoted $\alpha \vDash n$. The \textit{length} of a composition is $\ell(\alpha)=\ell$. A \textit{partition} $\lambda$ is a composition where the parts are weakly decreasing from left to right. A \textit{partition of size} $n$ is denoted $\lambda \vdash n$. For a composition $\alpha$, denote by $partition(\alpha)$ the partition corresponding to $\alpha$, obtained by ordering the parts in decreasing order \svw{from left to right.}
    The unique composition of size $0$ is denoted $\emptyset$. 
    Throughout this work, we assume all compositions are of size $n$, for some $n >0$, and typically denote compositions by $\alpha,\beta,\gamma$.

\begin{example}
    The composition $\alpha=(4,5,1)$ has size $10$ and length $3$. The partition $\lambda=(5,3,2,2)$ has size $12$ and length $4$. 
\end{example}

Given a composition $\alpha$, the \textit{(composition) diagram of} $\alpha$ is the left-justified array of cells in $\mathbb{N} \times \mathbb{N}$ such that the $i$-th row contains $\alpha_{i}$ cells. The rows are numbered $1,\dots,\ell$ from \textit{bottom to top}, and the columns are numbered $1,2,\dots$ from left to right. The cell in the $i$-th row and $j$-th column of the diagram is denoted $(i,j)$. 
We often abuse notation and refer to the diagram of $\alpha$ as $\alpha$, and vice versa. We say $\alpha$ has a \svw{\emph{row of length $k$}} if there exists some \svw{$\alpha_{i}=k$.} 
\begin{example}
The diagram of $\alpha=(4,5,1)$, with the cell $(2,4)$ indicated by $\ast$.
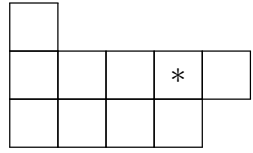
\begin{center}
    \begin{ytableau}
        \  \\
        \  & \  & \  & \ast & \  \\ 
        \  & \  & \  & \  \\
    \end{ytableau}
\end{center}
\vspace{0.1cm}
\end{example}

A composition $\alpha$ is a \textit{diving-board} if $\alpha$ has at most one row of length $k \geq 2$. If $\alpha \neq (1^n)$ is a diving-board, then the unique row of length $k \geq 2$ is referred to as the \textit{long row}. 
\svw{Given a diving-board, a} \textit{bottom-aligned hook} is either $\alpha = (1^n)$ or $\alpha \neq (1^n)$ where the long row occurs at index $i=1$. 
\svw{Similarly, a} \textit{top-aligned hook} is either $\alpha = (1^n)$ or $\alpha \neq (1^n)$ where the long row occurs at index $i=\ell(\alpha)$.

\begin{example}
\svw{Diving-board} diagrams, from left to right: the top-aligned hook $(1,1,5)$, the \svw{diving-board} of size $7$ where the long row occurs at index 2, $(1,4,1,1)$, and the bottom-aligned hook $(3,1,1)$.
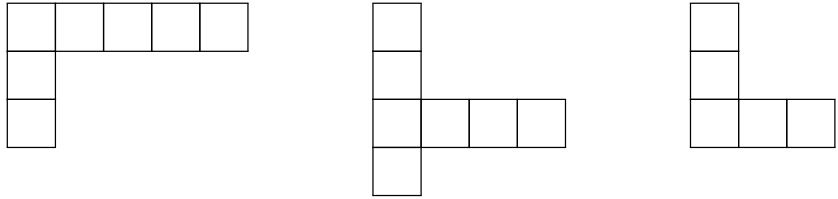
\begin{center}
  \ydiagram{5,1,1} \hspace{1.5cm}
        \ydiagram{1,1,4,1} \hspace{1.5cm}
        \ydiagram{1,1,3}
\end{center}
\vspace{0.1cm}
\end{example}

Given a composition $\alpha \vDash n$, a \textit{filling of} $\alpha$ is a map $T$ from the cells of the diagram of $\alpha$ to $\mathbb{N}$. \svw{We vizualize this by filling the cells of $\alpha$ with the numerical entries to obtain $T$, and call the result a \emph{tableau} of \emph{shape} $\alpha$.} We say $m = T(i,j)$ if the \svw{cell}  $(i,j)$ in $T$ contains the entry $m$. A \textit{standard filling} is a filling such that the map is bijective with image $[n]$. For a standard filling $T$ and entry $m \in [n]$, we say $R_m = i$ if the row index of $m$ is $i$, in other words, $m = T(i,j)$ for some $j$.

\begin{example}
A standard filling $T$ of the diagram $\alpha=(4,5,1)$. In particular, we have $ 1= T(3,1)$, and $R_8 = 2$.
\begin{center}
       \begin{ytableau}
        1 \\
        3 & 2 & 8 & 7 & 4 \\ 
        9 & 6 & 5 & 10 \\
    \end{ytableau}
\end{center}
\vspace{0.1cm}
\end{example}

There is a natural bijection between compositions of size $n$ and subsets of $[n-1]$. 
Given a composition $\alpha$, the set associated \svw{to} $\alpha = (\alpha_{1},\alpha_{2},\dots,\alpha_{\ell})$ is 
\[
\Set(\alpha) = \{ \alpha_{1},\alpha_{1}+\alpha_{2},\dots,\alpha_{1}+\alpha_{2}+\dots+\alpha_{\ell-1} \} \subseteq [n-1].\]
On the other hand, given a set $A = \{ A_{1} < A_{2} < \dots < A_{k} \} \svw{\subseteq} [n-1]$ written in strictly increasing order, the composition associated to $A$ is 
\[
\Comp(A) = (A_{1},A_{2}-A_{1},\dots,A_{k}-A_{k-1},n-A_{k}) \vDash n.
\]
Note the map $\Comp$ depends on the integer $n$, which is typically omitted from the notation. 
Given a composition $\alpha$, we always assume $\Set(\alpha)$ is sorted in increasing order.
For $n \geq 1$, the set associated to $(n) \vDash n$ is $\emptyset$.

\begin{example}
    The set associated to $\alpha=(4,5,1)$ is $\Set(\alpha)=\{ 4,9 \} \ME{\subseteq[9]}$. For the set $A = \{ 3,5,6,8 \} \ME{\subseteq [9]}$, the composition associated to $A$ is $\Comp(A) = (3,2,1,2,2)$. 
\end{example}

Similarly to the composition diagram of $\alpha$, we can visualize $\alpha$ as a sequence of dots and bars. In particular, $\alpha$ is associated to the sequence with $\alpha_{1}$ dots followed by a bar, then $\alpha_{2}$ dots followed by a bar, and so on.

\begin{example}
From left to right, the dots and bars diagram associated to the compositions $(4,5,1)$ and $(3,2,1,2,2)$.
\[
\bullet \bullet \bullet \, \bullet \mid \bullet \bullet \bullet \bullet \bullet \mid \bullet \hspace{1cm} 
        \bullet \bullet \,  \bullet \mid \bullet \, \bullet \mid \bullet \mid \bullet \, \bullet \mid \bullet \, \bullet
\]
\end{example}

We will be interested in three involutions on compositions. Let $\alpha \vDash n$, $\alpha = (\alpha_1, \alpha_2, \ldots, \alpha_\ell)$. The \textit{reverse} \svw{of}  $\alpha$ \svw{is the composition}  $\alpha^r=(\alpha_{\ell},\dots,\alpha_{2},\alpha_{1})$. The \textit{complement} of $\alpha$ \svw{is the composition $\alpha^c=\Comp(\Set(\alpha)^c)$,} where $\Set(\alpha)^c$ is the usual set complement in $[n-1]$. From the dots and bars interpretation of compositions, $\alpha^c$ is obtained from $\alpha$ by putting a bar between two dots if there was no bar, and removing a bar if there was one. The \textit{transpose} of $\alpha$ \svw{ is the composition $\alpha^t = (\alpha^c)^r = (\alpha^r)^c$.} 

\begin{example}
    From left to right, the composition $\alpha=(1,3,1,1)$, $\alpha^r=(1,1,3,1)$, $\alpha^c=(2,1,3)$, and $\alpha^t=(3,1,2)$.
\[
    \bullet \mid \bullet \bullet \bullet \mid \bullet \mid \bullet  \hspace{0.6cm}
         \bullet \mid \bullet \mid \bullet \bullet \bullet \mid \bullet  \hspace{0.6cm}
         \bullet \bullet \mid \bullet \mid \bullet \bullet \bullet 
         \hspace{0.6cm}
        \bullet \bullet \, \bullet \mid \bullet \mid \bullet \, \bullet
\]
\end{example}

Let $\alpha=(\alpha_{1},\alpha_{2},\ldots,\alpha_{\ell(\alpha)})$ and $\beta=(\beta_{1},\beta_{2},\ldots,\beta_{\ell(\beta)})$ be two compositions of size $n$. We say $\beta$ is less than $\alpha$ in \textit{lexicographic order}, denoted $\beta \leq_{l} \alpha$, if $\alpha=\beta$ or $\beta_{i}<\alpha_{i}$ where $i$ is the minimal index such that $\beta_{i} \neq \alpha_{i}$. 
We say $\beta$ is less than $\alpha$ in \textit{dominance order}, denoted $\beta \leq_{d} \alpha$, if for all $1 \leq k \leq \min \{ \ell(\alpha),\ell(\beta) \}$, 
$
\sum_{i=1}^k \beta_{i} \leq \sum_{i=1}^k \alpha_{i}.
$
Equivalently, if $\Set(\alpha) = \{ A_{1} < A_{2} < \dots < A_{\ell(\alpha)-1} \}$ and $\Set(\beta) = \{ B_{1} < B_{2} < \dots < B_{\ell(\beta)-1} \}$, then $\beta \leq_{d} \alpha$ if and only if $\ell(\alpha) \leq \ell(\beta)$ and $B_{k} \leq A_{k}$ for all $1 \leq k \leq  \ell(\alpha)-1$.  
Lastly, we say $\beta$ is less than $\alpha$ in \textit{refinement order}, denoted $\beta \preceq \alpha$, if the parts of $\alpha$ can be obtained in order by adding adjacent parts of $\beta$ in order.

\begin{example}
The lexicographic order on compositions of size $4$ is the following chain:
\[
 (1,1,1,1) \leq_{l} (1,1,2) \leq_{l} (1,2,1) \leq_{l} (1,3) \leq_{l} (2,1,1) \leq_{l} (2,2) \leq_{l} (3,1) \leq_{l} (4).
\]
The dominance order on compositions of size $5$ contains the following chain:
$$
(1,1,2,1) \leq_{d} (1,3,1) \leq_{d} (2,3) \leq_{d} (4,1) \leq_{d} (5),
$$
and the refinement order on compositions of size $5$ contains the following chain:
\[
(1,1,1,1,1) \preceq (1,1,2,1) \preceq (1,1,3) \preceq (2,3) \preceq (5).
\]
\end{example}

\subsection{Algebras of symmetric, quasisymmetric functions}\label{ch:sym_bg}
We define the  \svw{algebra} of quasisymmetric functions, $\QSym$, as a graded algebra, where each graded component is spanned by a basis indexed by compositions, called the \svw{\emph{monomial quasisymmetric basis} of $\QSym$, which consists of all monomial quasisymmetric functions.} Let $\alpha \vDash n$ throughout this section.

    Suppose $n \geq 1$ and $\alpha$ has length $\ell$. Then the \svw{\emph{monomial quasisymmetric function}} indexed by $\alpha$ is 
\[
M_{\alpha}=\sum x_{i_{1}}^{\alpha_{1}}x_{i_{2}}^{\alpha_{2}}\cdots x_{i_{\ell}}^{\alpha_{\ell}},
\]
where the sum is taken over all positive $\ell$-tuples $i_{1} < i_{2} < \dots < i_{\ell}$. Define $M_{\emptyset}=1$. 

\begin{example}
For $(1,3),(3,1) \vDash 4$,
\begin{align*} 
M_{(3,1)} &= x_{1}^3x_{2} + x_{1}^3x_{3} + x_{1}^3x_{4} + \svw{\cdots} + x_{2}^3x_{3} + x_{2}^3x_{4} + \svw{\cdots,}\\
M_{(1,3)} &= x_{1}x_{2}^3 + x_{1}x_{3}^3 + x_{1}x_{4}^3 + \svw{\cdots} + x_{2}x_{3}^3 + x_{2}x_{4}^3 + \svw{\cdots.}
\end{align*}
\end{example}

\svw{For $n \geq 0$,} 
\[
\QSym^n = \Span_{\mathbb{Z}} \{ M_{\alpha} \mid \alpha \vDash n \},
\] and the \textit{algebra of quasisymmetric functions} is 
\[
\QSym= \bigoplus_{n \geq 0} \QSym^n.
\]

Next, we define the \svw{\emph{fundamental basis} of $\QSym$, which consists of all fundamental quasisymmetric functions.} The \svw{\emph{fundamental quasisymmetric function}} indexed by $\alpha$ is
\[
F_{\alpha}=\sum_{\beta \preceq \alpha} M_{\beta}.
\]
For $n \geq 0$, \svw{it is well known that} the set $\{ F_{\alpha} \mid \alpha \vDash n \}$ forms a $\mathbb{Z}$-basis of $\QSym^n$. 

\begin{example}\label{ex:exp_of_F_13}
The expansion of $F_{(1,3)}$ into the monomial \svw{quasisymmetric} basis is
\[
F_{(1,3)} = M_{(1,3)}+M_{(1,2,1)} + M_{(1,1,2)} + M_{(1,1,1,1)}.
\]
\end{example}

Analogously to the definition of $\QSym$, we define the  \svw{algebra} of symmetric functions, $\Sym$, as a graded algebra, where each graded component is spanned by a basis indexed by partitions, called the \svw{\emph{monomial symmetric basis}, which consists of all monomial symmetric functions.}

    For $n \geq 1$, let $\lambda \vdash n$ be a partition. Then, the \svw{\emph{monomial symmetric function}} indexed by $\lambda$ is 
\[
m_{\lambda}=\sum_{partition(\alpha)=\lambda} M_{\alpha}.
\]
Define $m_{\emptyset}=1$.

\begin{example}\label{ex:exp_of_m_31}
    For $(3,1) \vdash 4$,
\[
m_{(3,1)} = M_{(3,1)} + M_{(1,3)}.
\]
\end{example}

\svw{For $n \geq 0$,} 
    \[
    \Sym^n = \Span_{\mathbb{Z}} \{ m_{\lambda} \mid \lambda \vdash n \},
    \] and the \textit{algebra of symmetric functions} is
\[
\Sym= \bigoplus_{n \geq 0} \Sym^n.
\]

Next, we define the renowned \svw{\emph{Schur basis} of $\Sym$, which consists of all Schur functions.} We do so as an expansion over the fundamental quasisymmetric functions, as this will mirror the expansions of the dual immaculate functions in the fundamental basis in Section \ref{ch:dI_bg}. Compare point 2 of the following definition with Definition \ref{def:imm_tabs} of immaculate tableaux.

\begin{definition}\label{def:young_tabs}
For $n\geq 1$, let $\lambda \vdash n$ be a partition. A \textit{standard Young tableau} $T$ of \ME{shape} $\lambda$ is a bijective filling of the cells of the diagram of $\lambda$ with the integers $\{ 1,2,\ldots, n \}$ such that 
\begin{enumerate}
    \item the entries in each row are increasing, when read from \textit{left to right},
    \item the entries in \textit{each} column are increasing when read from \textit{bottom to top}.
\end{enumerate}
Given a standard Young tableau $T$ of shape $\lambda$, we define the \textit{descent set} of $T$ \svw{to be}
\[
    \Des(T) = \{ i \mid i+1 \text{ appears strictly above } i \} \subseteq [n-1].
\]
\end{definition}

\begin{example}
    \svw{A standard} Young tableau $T$ of shape $(5,3,2,2) \vdash 12$ with $\Des(T)=\{1,4,6,7,10,11 \}$.
    \begin{center}
           \begin{ytableau}
        8 & 12 \\
        7 & 11 \\
        2 & 5 & 6 \\ 
        1 & 3 & 4 & 9 & 10 
    \end{ytableau}
    \end{center}
    \vspace{0.1cm}
\end{example}

Then, for $n\geq 1$ and $\lambda \vdash n$ a partition, the \svw{\textit{Schur function}} indexed by $\lambda$ is 
\[
    s_{\lambda} = \sum_{T} F_{\Comp(\Des(T))},
\]
where the sum is over all standard Young tableaux $T$ of shape \svw{$\lambda$.} Define $s_{\emptyset}=1$. 
For $n \geq 0$, \svw{it is well known that} the set $ \{ s_{\lambda} \mid \lambda \vdash n \}$ forms a $\mathbb{Z}$-basis of $\Sym^n$. 

\begin{example}\label{ex:exp_of_s_32}
We have $s_{(3,2)}= F_{(3,2)} + F_{(2,3)} + F_{(2,2,1)} + F_{(1,3,1)} + F_{(1,2,2)}$, from the following standard Young tableaux. 
\center{
    \begin{ytableau}
        4 & 5 \\ 
        1 & 2 & 3
    \end{ytableau} \hspace{0.5cm}
    \begin{ytableau}
        3 & 4 \\ 
        1 & 2 & 5 
    \end{ytableau} \hspace{0.5cm}
    \begin{ytableau}
        3 & 5 \\ 
        1 & 2 & 4 
    \end{ytableau} \hspace{0.5cm}
    \begin{ytableau}
        2 & 5 \\ 
        1 & 3 & 4
    \end{ytableau} \hspace{0.5cm}
    \begin{ytableau}
        2 & 4 \\ 
        1 & 3 & 5
    \end{ytableau} 
    }
    \vspace{0.1cm}
\end{example}
\vspace{0.2cm}

Consider the algebra $\Sym$ or $\QSym$ and let $f$ be an element of the algebra. Let $\mathcal{B}$ be a basis of the algebra, and $B$ a basis element. Then, the notation $[B]f$ denotes the coefficient of $B$ in the expansion of $f$ in $\mathcal{B}$. For example, 
$[M_{(1,1,2)}]F_{(1,3)}=1$ by Example \ref{ex:exp_of_F_13} and $[F_{(1,1,1,1,1)}]s_{(3,2)}=0$ by Example \ref{ex:exp_of_s_32}.

Recall the involutions on compositions, defined in Section \ref{ch:comp_bg}. These correspond to involutive automorphisms of $\QSym$, which we state here in terms of the fundamental basis. 

\begin{definition}\cite{luotoIntroductionQuasisymmetricSchur2013a}\label{def:qsym_involutions}
Define $\rho, \psi, \omega : \QSym \to \QSym$ by
\begin{align*}
    \rho(F_{\alpha}) &= F_{\alpha^r}, \\
    \psi(F_{\alpha}) &= F_{\alpha^c}, \\
    \omega(F_{\alpha}) &= F_{\alpha^t},
\end{align*}
and \ME{extending} linearly. 
\end{definition}

\begin{example}
    We have 
    \[
    \psi(s_{(3,2)}) = F_{(1,1,2,1)} + F_{(1,2,1,1)} + F_{(1,2,2)} + F_{(2,1,2)} + F_{(2,2,1)}.
    \]
\end{example}

Note since $\alpha^t=(\alpha^r)^c = (\alpha^c)^r$, we have $\omega = \rho \circ \psi = \psi \circ \rho$. Restricting to $\Sym$, $\rho$ restricts to the identity, while $\psi,\omega$ restrict to the well-known conjugating automorphism $\omega:\Sym \to \Sym$ \cite{luotoIntroductionQuasisymmetricSchur2013a}.
Note that for a partition $\lambda$, $\omega(s_\lambda) = s_{\lambda'}$ where $\lambda'$ is the partition corresponding to the diagram of $\lambda$, flipped along the $y=x$ line \cite[p333]{StanleyEnumComb}. Since $\lambda \mapsto \lambda'$ is an involution, it follows that the Schur basis is invariant under $\omega$, i.e. $\{ s_{\lambda} \}_{\lambda \vdash n} = \{ \omega(s_{\lambda}) \}_{\lambda \vdash n}$. Furthermore, if $\lambda = (k,1^{\ell-1})$ is a bottom-aligned hook, then $\lambda'=(\ell,1^{k-1})=\lambda^t$, so $\omega(s_{\lambda}) = s_{\lambda^t}$.

\subsection{Dual immaculate basis and variants}\label{ch:dI_bg}
\svw{In this subsection} we will define the {dual immaculate basis} of $\QSym$, \svw{which consists of all dual immaculate functions, each of which can be seen} as a generating function of immaculate tableaux, which generalize Young tableaux. We give an expansion over the monomial quasisymmetric basis, as well as the \svw{fundamental basis.}  \svw{Throughout, the first column refers to the leftmost column.}

\begin{definition}\label{def:imm_tabs}
    Let $\alpha \vDash n$ and let $T$ be a filling of $\alpha$. Then the \textit{content} of $T$ is a sequence of nonnegative integers $\beta$ such that $\beta_{i}$ is the number of entries $i$ in $T$, \svw{and is denoted $\content(T)$.} 
    \begin{enumerate}
        \item A \textit{(semi-standard) immaculate tableau} \svw{$T$ of shape} $\alpha$ is a filling of $\alpha$ such that
         
        \begin{itemize}
            \item the entries are \textit{strictly} increasing from bottom to top in the \textit{first} column, and
            \item \textit{weakly} increasing from \textit{left to right} within each row. 
        \end{itemize}
        Denote the set of immaculate tableaux by $\SSIT$, the set of immaculate tableaux of shape $\alpha$ by $\SSIT(\alpha)$, and the set of immaculate tableaux of shape $\alpha$ and content $\beta$ by $\SSIT(\alpha,\beta)$. 
	
\item A \textit{row-strict (semi-standard) immaculate tableau} \svw{$T$ of shape} $\alpha$ is a filling of $\alpha$ such that 

    \begin{itemize}
        \item the entries are \textit{weakly} increasing from bottom to top in the \textit{first} column, and
	\item \textit{strictly} increasing from \textit{left to right} within each row. 
    \end{itemize}
    Denote the set of row-strict immaculate tableaux by $\RSSIT$, the set of row-strict immaculate tableaux of shape $\alpha$ by $\RSSIT(\alpha)$, and the set of row-strict immaculate tableaux of shape $\alpha$ and content $\beta$ by $\RSSIT(\alpha,\beta)$. 
	
\item A \textit{(semi-standard) reverse immaculate tableau} \svw{$T$ of shape} $\alpha$ is a filling of $\alpha$ such that 

\begin{itemize}
    \item the entries are \textit{strictly} increasing from bottom to top in the \textit{first} column, and
	\item \textit{weakly} increasing from \textit{right to left} within each row. 
\end{itemize}
Denote the set of reverse immaculate tableaux by $\SSRIT$, the set of reverse immaculate tableaux of shape $\alpha$ by $\SSRIT(\alpha)$, and the set of reverse immaculate tableaux of shape $\alpha$ and content $\beta$ by $\SSRIT(\alpha,\beta)$. 
	
\item A \textit{row-strict (semi-standard) reverse immaculate tableau} \svw{$T$ of shape} $\alpha$ is a filling of $\alpha$ such that 

\begin{itemize}
    \item the entries are \textit{weakly} increasing from bottom to top in the \textit{first} column, and
	\item \textit{strictly} increasing from \textit{right to left} within each row. 
\end{itemize}
Denote the set of row-strict reverse immaculate tableaux by $\RSSRIT$, the set of row-strict reverse immaculate tableaux of shape $\alpha$ by $\RSSRIT(\alpha)$, and the set of row-strict reverse immaculate tableaux of shape $\alpha$ and content $\beta$ by $\RSSRIT(\alpha,\beta)$. 

\end{enumerate}
We refer to the tableaux in $\SSIT \cup \RSSIT \cup \SSRIT \cup \RSSRIT$ as \textit{variants of immaculate tableaux}. 
\end{definition}

\begin{table}[h!]
\begin{center}
\begin{tabular}{|c|c|c|c|}
\hline
Variant& First column & Row\\
\hline
Immaculate & strict B to T & weak L to R\\
\hline
Row-strict & weak B to T & strict L to R\\
\hline
Reverse & strict B to T & weak R to L\\
\hline
Row-strict reverse & weak B to T & strict R to L\\
\hline
\end{tabular}
\end{center}
\caption{Variants of immaculate tableaux.}
\end{table}

\begin{example}\label{ex:semistd_tab_ex}
\svwrev{Instances of the} four variants of immaculate tableaux of shape $\alpha = (4,5,1)$ and content $\beta=(2,2,1,1,2,1,1)$.
       \begin{align*}
    T_1 &= 
    \begin{ytableau}
        5 \\
        2 & 3 & 4 & 5 & 7 \\ 
        1 & 1 & 2 & 6 
    \end{ytableau} \in \SSIT(\alpha,\beta)\\
    T_2 &= 
    \begin{ytableau}
        3 \\
        1 & 2 & 4 & 5 & 6 \\ 
        1 & 2 & 5 & 7
    \end{ytableau} \in \RSSIT(\alpha,\beta)\\
    T_3 &= 
    \begin{ytableau}
        7 \\
        6 & 4 & 2 & 2 & 1 \\ 
        5 & 5 & 3 & 1
    \end{ytableau} \in \SSRIT(\alpha,\beta)\\
    T_4 &= 
    \begin{ytableau}
        7 \\
        6 & 5 & 3 & 2 & 1 \\ 
        5 & 4 & 2 & 1
    \end{ytableau} \in \RSSRIT(\alpha,\beta)\\
    \end{align*}
\end{example}

\vspace{0.2cm}

For $\alpha \vDash n$, the set of \ME{\textit{standard immaculate tableaux}} of shape $\alpha$ is \svw{$\SIT(\alpha)=\SSIT(\alpha,(1^n))$.} The set of \ME{\textit{standard reverse immaculate tableaux}} of shape $\alpha$ is \svw{$\SRIT(\alpha)=\SSRIT(\alpha,(1^n))$.} Note $\RSSIT(\alpha,(1^n))=\SIT(\alpha)$ and $\RSSRIT(\alpha,(1^n)) = \SRIT(\alpha)$. \svw{Denote the set of all standard immaculate tableaux by $\SIT$ and of all reverse standard immaculate tableaux by $\SRIT$.}

Let $T \in \SIT(\alpha) \cup \SRIT(\alpha)$.
\begin{enumerate}
    \item Define the $\mathfrak{S}^*$-\textit{descent set} \svw{of} $T$ by 
    \[
\Des_{\mathfrak{S}^*}(T) = \{ i \mid i+1 \text{ is strictly above }i \text{ in } T \} \subseteq [n-1].
\]
The $\mathfrak{S}^*$-\textit{descent composition} \ME{of} $T$ is $\Comp(\Des_{\mathfrak{S}^*}(T))$. 
\item Define the $R\mathfrak{S}^*$-\textit{descent set} \svw{of} $T$ by 
\[
\Des_{R\mathfrak{S}^*}(T) =  \{ i \mid i+1 \text{ is weakly below }i \text{ in } T \} \subseteq [n-1].
\]
The $R\mathfrak{S}^*$-\textit{descent composition} \ME{of} $T$ is $\Comp(\Des_{R\mathfrak{S}^*}(T))$. 
\end{enumerate}

We typically denote the descent composition of a tableau by $\beta$. Let $\beta \vDash n$, and $\ast$ denote either $\mathfrak{S}^*$ or $R\mathfrak{S}^*$. Let $\SIT_{\ast}(\alpha;\beta)$ denote the subset of $\SIT(\alpha)$ with $\ast$-descent composition $\beta$, and let $\SRIT_{\ast}(\alpha;\beta)$ denote the subset of $\SRIT(\alpha)$ with $\ast$-descent composition $\beta$. 
We typically abuse the notation and use $\Des_{\ast}$ to denote both the descent set and descent composition, which should be clear from the context.

For a standard (reverse) immaculate tableau $T$ of shape $\alpha \vDash n$, note that $\Des_{R\mathfrak{S}^*}(T) = (\Des_{\mathfrak{S}^*}(T))^c$ as subsets of $[n-1]$, and so 
\[
\Comp(\Des_{R\mathfrak{S}^*}(T)) = (\Comp(\Des_{\mathfrak{S}^*}(T)))^c.
\]

\begin{remark}\label{rem:descent_comp_as_runs}
The $\mathfrak{S}^*$-descent composition of a standard (reverse) immaculate tableau $T$ corresponds to the lengths of the maximal runs formed by reading the entries in the order $1,2,\dots,n$ weakly down rows. For example, the maximal runs in $S_1$ in Example \ref{ex:std_tab_ex} are $1 \, 2,3 \, 4,5 \, 6,7 \, 8 \, 9,10$, and so the $\mathfrak{S}^*$-descent composition of $S_1$ is $(2,2,2,3,1)$. The integers $5,6$ form a run since the row index of $5$ is $R_5=2$, while $R_6=2$, so $6$ is weakly below $5$. The run $5,6$ is maximal since $5$ is strictly above $4$, i.e. $R_4 = 1 < R_5 = 2$, and $7$ is strictly above $6$, i.e. $R_6=2 < R_7 = 3$. 

Similarly, the $R\mathfrak{S}^*$-descent composition of a standard (reverse) immaculate tableau $T$ corresponds to the lengths of the maximal runs formed by reading the entries in the order $1,2,\dots,n$ strictly up rows. 
\end{remark}

\begin{example}\label{ex:std_tab_ex}
\svwrev{Instances of the} four variants of standard immaculate tableaux of shape $\alpha = (4,5,1)$, with the associated descent compositions.
   \begin{align*}
    S_1 &= 
    \begin{ytableau}
        7 \\
        3 & 5 & 6 & 8 & 10 \\ 
        1 & 2 & 4 & 9 
    \end{ytableau} \in \SIT_{\dI^*}(\alpha;(2,2,2,3,1)). \\
    S_2 &= 
    \begin{ytableau}
        5 \\
        2 & 4 & 6 & 8 & 9 \\ 
        1 & 3 & 7 & 10
    \end{ytableau} \in \SIT_{R\dI^*}(\alpha;(2,3,1,2,1,1)).\\
    S_3 &= 
    \begin{ytableau}
        10 \\
        9 & 6 & 4 & 3 & 1 \\ 
        8 & 7 & 5 & 2
    \end{ytableau} \in \SRIT_{\dI^*}(\alpha;(2,3,3,1,1)).\\
    S_4 &= 
    \begin{ytableau}
        10 \\
        9 & 8 & 5 & 4 & 2 \\ 
        7 & 6 & 3 & 1
    \end{ytableau} \in \SRIT_{R\dI^*}(\alpha;(2,2,1,1,2,2)).
\end{align*}
\end{example}

Define the map $\flip : \SIT \cup \SRIT \to \SIT \cup \SRIT$ as follows. Let $T \in \SIT \cup \SRIT$ have shape $\alpha \vDash n$. Obtain $\flip(T)$ from $T$ by reversing the \svw{order of the} rows of $T$, and replacing each entry $i$ with $n+1-i$. Note $\flip$ is an involution.

\begin{lemma}\label{lem:flip_reverses_content}\cite[Lemma 3.5]{masonLiftingDualImmaculate2021}
    Let $\alpha,\beta \vDash n$ and $\ast \in \{ \mathfrak{S}^*,R\mathfrak{S}^* \}$.
\begin{enumerate}
    \item \svw{Let} $T \in \SIT_{\ast}(\alpha;\beta)$. Then $\flip(T) \in \SRIT_{\ast}(\alpha^r;\beta^r)$.
\item \svw{Let} $T \in \SRIT_{\ast}(\alpha;\beta)$. Then $\flip(T) \in \SIT_{\ast}(\alpha^r;\beta^r)$.
\end{enumerate}
\end{lemma}

\begin{example}
\svw{The} flip map applied to the standard tableaux $S_1$ and $S_4$ in Example \ref{ex:std_tab_ex}.
    \begin{align*}
    \flip(S_1) &= 
    \begin{ytableau}
        10 & 9 & 7 & 2 \\
        8 & 6 & 5 & 3 & 1 \\ 
        4
    \end{ytableau} \in \SRIT_{\dI^*}((1,5,4);(1,3,2,2,2)). \\
    \flip(S_4) &= 
    \begin{ytableau}
        4 & 5 & 8 & 10 \\
        2 & 3 & 6 & 7 & 9 \\ 
        1 
    \end{ytableau} \in \SIT_{R\dI^*}((1,5,4);(2,2,1,1,2,2)).
\end{align*}  
\end{example}

We are now ready to define the \svw{\emph{dual immaculate basis}} of $\QSym$, \svw{which consists of all dual immaculate functions}, as well as its variants \svw{that will be named analogously.} 

\begin{definition}\label{defn:dI_over_M}
For $n \geq 1$, let $\alpha \vDash n$ be a composition. 
\begin{enumerate}
    \item \cite[Proposition 3.36]{bergLiftSchurHallLittlewood2014} The \textit{dual immaculate function} \svw{indexed by} $\alpha$ is 
    \[
\mathfrak{S}^*_{\alpha} = \sum_{\beta \vDash n} K_{\alpha,\beta} M_{\beta},
\]
where $K_{\alpha,\beta} = |\SSIT(\alpha,\beta)|.$

\item \svw{\cite[p12]{nieseRowstrictDualImmaculate2023}} The \textit{row-strict dual immaculate function} \svw{indexed by} $\alpha$ is \[
R\mathfrak{S}^*_{\alpha} = \sum_{\beta \vDash n} K_{\alpha,\beta}^{R\mathfrak{S}^*} M_{\beta},
\]
where $K_{\alpha,\beta}^{R\mathfrak{S}^*} = |\RSSIT(\alpha,\beta)|.$

\item \cite[Proposition 3.2.13]{daughertySchurlikeBasesTheir2024} The \textit{reverse dual immaculate function} \svw{indexed by} $\alpha$ is \[
\BdI^*_{\alpha} = \sum_{\beta \vDash n} K_{\alpha,\beta}^{\BdI^*} M_{\beta},
\]
where $K_{\alpha,\beta}^{\BdI^*} = |\SSRIT(\alpha,\beta)|.$

\item \cite[Proposition 3.2.32]{daughertySchurlikeBasesTheir2024} The \textit{row-strict reverse dual immaculate function} \svw{indexed by} $\alpha$ is \[
R\BdI^*_{\alpha} = \sum_{\beta \vDash n} K_{\alpha,\beta}^{R\BdI^*} M_{\beta},
\]
where $K_{\alpha,\beta}^{R\BdI^*} = |\RSSRIT(\alpha,\beta)|.$

\end{enumerate}
Define $\mathfrak{S}^*_{\emptyset} = R\mathfrak{S}^*_{\emptyset} = \BdI^*_{\emptyset} = R\BdI^*_{\emptyset }= 1$. 
\end{definition}

\begin{example}
We have $R\mathfrak{S}^*_{(2,1)}= M_{(2,1)}+ M_{(1,2)}+ 2M_{(1,1,1)}$, from the following \svw{row-strict semi-standard}  immaculate tableaux.
\center{
        \begin{ytableau}
        1 \\ 
        1 & 2
    \end{ytableau}\hspace{0.8cm}
    \begin{ytableau}
        2 \\ 
        1 & 2
    \end{ytableau} \hspace{0.8cm}
     \begin{ytableau}
        2 \\ 
        1 & 3
    \end{ytableau} \hspace{0.8cm}
    \begin{ytableau}
        3 \\ 
        1 & 2
    \end{ytableau} 
    }
    \vspace{0.1cm}
\end{example}

\begin{proposition}\cite[Proposition 3.15]{bergLiftSchurHallLittlewood2014}\label{prop:prop315}
Let $\alpha,\beta \vDash n$. Then $K_{\alpha,\alpha}=1$. Moreover, if $K_{\alpha,\beta}>0$, then $\beta \leq_{l} \alpha$. 
\end{proposition}

Alternatively, we give an expansion over the fundamental basis using descent compositions \svw{that will also be useful to us.}

\begin{proposition}\label{prop:dI_over_F}\label{def:dI_fcns}
For $n \geq 1$, let $\alpha \vDash n$ be a composition. 
\begin{enumerate}
    \item \cite[Proposition 3.37]{bergLiftSchurHallLittlewood2014} The dual immaculate function indexed by $\alpha$ is 
\[
\mathfrak{S}^*_{\alpha} = \sum_{T \in \SIT(\alpha)} F_{\Des_{\mathfrak{S}^*}(T)}.
\]

\item \cite[Theorem 3.6]{nieseRowstrictDualImmaculate2023} The row-strict dual immaculate function indexed by $\alpha$ is \[
R\mathfrak{S}^*_{\alpha} = \sum_{T \in \SIT(\alpha)} F_{\Des_{R\mathfrak{S}^*}(T)}.
\]

\item \svw{\cite[Equation 3.1]{masonLiftingDualImmaculate2021}} The reverse dual immaculate function indexed by $\alpha$ is \[
\BdI^*_{\alpha} = \sum_{T \in \SRIT(\alpha)} F_{\Des_{\mathfrak{S}^*}(T)}.
\]

\item \cite[Proposition 3.2.30]{daughertySchurlikeBasesTheir2024} The row-strict reverse dual immaculate function indexed by $\alpha$ is \[
R\BdI^*_{\alpha} = \sum_{T \in \SRIT(\alpha)} F_{\Des_{R\mathfrak{S}^*}(T)}.
\]

\end{enumerate} 
\end{proposition}

\begin{example}
We have $\mathfrak{S}^*_{(2,3)}= F_{(2,3)}+ F_{(1,4)}+ F_{(1,3,1)}+ F_{(1,2,2)}$, from the following standard immaculate tableaux. 
\center{
        \begin{ytableau}
        3 & 4 & 5 \\ 
        1 & 2
    \end{ytableau}\hspace{0.8cm}
    \begin{ytableau}
        2 & 3 & 4 \\ 
        1 & 5
    \end{ytableau} \hspace{0.8cm}
     \begin{ytableau}
        2 & 3 & 5 \\ 
        1 & 4
    \end{ytableau} \hspace{0.8cm}
    \begin{ytableau}
        2 & 4 & 5 \\ 
        1 & 3
    \end{ytableau} 
    }
    \vspace{0.1cm}
\end{example}

\begin{proposition}\label{prop:variants_dI_as_images}
Let $\alpha \vDash n$. \svw{Then}
\begin{align}
R\mathfrak{S}^*_{\alpha} &= \psi(\mathfrak{S}^*_{\alpha}), \label{eq:dI_psi} \\
\BdI^*_{\alpha} &= \rho(\mathfrak{S}^*_{\alpha^r}), \label{eq:dI_rho} \\
R\BdI^*_{\alpha} &= \omega(\mathfrak{S}^*_{\alpha^r}). \label{eq:dI_omega}
\end{align}
Moreover, the sets $\{ \mathfrak{S}^*_{\alpha} \}_{\alpha \vDash n}, \{ R\mathfrak{S}^*_{\alpha} \}_{\alpha \vDash n}, \{ \BdI^*_{\alpha} \}_{\alpha \vDash n} , \{ R\BdI^*_{\alpha} \}_{\alpha \vDash n}$ all form $\mathbb{Z}$-bases of $\QSym^n$.
\end{proposition}

\begin{proof} The dual immaculate functions of degree $n$ form a basis of $\QSym^n$ \svw{by \cite[Section 3.7]{bergLiftSchurHallLittlewood2014}}. 
Equations (\ref{eq:dI_psi}), (\ref{eq:dI_rho}), (\ref{eq:dI_omega}) are \cite[Theorem 3.8]{nieseRowstrictDualImmaculate2023}, \cite[Theorem 3.2.12]{daughertySchurlikeBasesTheir2024},
\cite[Theorem 3.2.33]{daughertySchurlikeBasesTheir2024}, respectively.
\svw{Hence,} the variants of dual immaculate functions each form bases of $\QSym^n$ since $\psi,\rho,\omega$ are involutions.
\end{proof}

Finally, we end with a result classifying when the dual immaculate functions are actually symmetric. \svw{The statements are the case $\beta = \emptyset$ of \cite[Theorem 3.4]{EsipovaSymmQsymSchur} and \cite[Corollary 3.6]{EsipovaSymmQsymSchur}, respectively.}

\begin{proposition} \label{prop:dI_symmetric}
Let $\alpha \vDash n$. Then $\mathfrak{S}^*_{\alpha} \in \Sym$ if and only if $\alpha$ is bottom-aligned hook. In this case, 
\[
    \dI^*_\alpha = s_\alpha.
    \]
\end{proposition}

\section{Compositions and complement}\label{ch:comps_and_complement}
In this section we study the complement involution on compositions. Our main goal will be to show it is order-reversing with respect to the dominance order, which will be useful to us later in the study of the involution $\psi$ on the dual immaculate functions. 
Our main tools will be the dots and bars diagram for compositions, and the equivalent definition for the dominance order on subsets of $[n-1]$. We begin with a result on how the length of a composition changes after taking the complement.

\begin{lemma}\label{lem:len_complement}
    Let $\alpha \vDash n$ with $\ell(\alpha)=\ell$.
    Then $\ell(\alpha^c) = n+1 - \ell$. In particular, for $\alpha,\beta \vDash n$, $\ell(\alpha) \leq \ell(\beta)$ if and only if $\ell(\alpha^c) \geq \ell(\beta^c)$.
\end{lemma}

\begin{proof}
    Associate a composition to its dots and bars diagram. Note the length of a composition is one more than the number of bars. Let $b$ be the number of bars in $\alpha$ such that $b+1 = \ell$. Then the number of bars in $\alpha^c$ is $n-(b+1)$. Thus, 
    \[
    \ell(\alpha^c) = n-(b+1) + 1 = n+1 - \ell.
    \]
\end{proof}

Let $A= \{A_1,A_2,\ldots,A_\ell \},B=\{B_1,B_2,\ldots,B_p \} \subseteq [n]$. We say $B$ is \textit{pairwise less than} $A$ if $\ell \leq p$ and $B_k \leq A_k$ for all $1 \leq k \leq \ell$.

\begin{proposition}\label{prop:dom_complement}
    The complement operation on compositions is an order-reversing involution with respect to the dominance order. In other words, 
for $\alpha,\beta \vDash n$, $\beta \leq_d \alpha$ if and only if $\beta^c \geq_{d} \alpha^c$. 
\end{proposition}

\begin{proof}
    Suppose $\beta \leq_{d} \alpha$, with corresponding sets $B,A \subseteq [n-1]$, respectively. Using the equivalent definition of the dominance order, $B$ is pairwise less than $A$. To show $\alpha^c \leq_d \beta^c$, we need to show $A^c$ is pairwise less than $B^c$. 
    By Lemma \ref{lem:len_complement}, $\ell(\alpha^c) \geq \ell(\beta^c)$ since $\ell(\alpha) \leq \ell(\beta)$, as required by the definition of dominance order. Observe that since $B$ is pairwise less than $A$, for all $m \in [n-1]$ \svw{we have} $|\{ a \in A \mid a \leq m \}| \leq | \{b \in B \mid b \leq m \} |$.

If $A^c_k \leq B^c_k$ for all $1 \leq k \leq \ell(\beta^c) - 1$, then we are done. Otherwise, let $k$ be the minimal index such that $B^c_{k} < A^c_{k}$ and denote $m=B^c_{k}$. Then $\{ \svw{\tilde{b}} \in B^c \mid \svw{\tilde{b}} \leq m \} = \{ B^c_{1} , \dots, B^c_{k}\}$, while $\{ \svw{\tilde{a}} \in A^c \mid \svw{\tilde{a}} \leq m \} = \{ A^c_{1},\dots,A^c_{k-1} \}$ by the minimality of $k$. Thus,
\begin{align*}
|\{ b \in B \mid b \leq m \}| &= m - | \{ \svw{\tilde{b}} \in B^c \mid \svw{\tilde{b}} \leq m \} | \\ 
&= m - (k-1) - 1 \\
&= | \{ a \in A \mid a \leq m \}| - 1.
\end{align*}
But this contradicts our earlier observation that $|\{ a \in A \mid a \leq m \}| \leq | \{b \in B \mid b \leq m \} |$. 
\end{proof}

\begin{example}
    Let $\alpha = (4,5,1)$ and $\beta=(3,6,1)$. Then $\beta \leq_d \alpha$ and \svw{$\beta^c = (1,1,2,1^4, 2)) \geq_d (1,1,1, 2,1^3,2) = \alpha^c$. }
    \begin{align*}
        \alpha &= \bullet \bullet \bullet \, \bullet \mid \bullet \bullet \bullet \bullet \bullet \mid \bullet, \hspace{1cm} \alpha^c = \bullet \mid \bullet \mid \bullet \mid \bullet \, \bullet \mid \bullet \mid \bullet \mid \bullet \mid \bullet \, \bullet \\
        \beta &= \bullet \bullet \bullet \mid \bullet \bullet \bullet \bullet \bullet \, \bullet \mid \bullet, 
        \hspace{1cm} \beta^c = \bullet \mid \bullet \mid \bullet \, \bullet \mid \bullet \mid \bullet \mid \bullet \mid \bullet \mid \bullet \, \bullet 
    \end{align*}
\end{example}

We finish with an observation that the transpose operation preserves the property of being a top- \svw{(or bottom-)aligned} hook.

\begin{proposition}\label{prop:transpose_preserves_hooks}
Let $\alpha \vDash n$. 
    Then $\alpha$ is a \svw{top- (bottom-)aligned} hook if and only if $\alpha^t$ is a \svw{top- (bottom-)aligned} hook.
\end{proposition}

\begin{proof}
If $\alpha=(1,\dots,1,k)$, then $\alpha^c = (\ell(\alpha), 1^{k-1})$ \svw{so $\alpha^t = (1^{k-1}, \ell(\alpha))$, and this is reversible.} The result for top-aligned hooks follows by taking the reverse \svw{of $\alpha$.} 
\end{proof}

\section{\svw{Leading terms of variants of dual immaculate functions}}\label{ch:leading_terms}
Let $\alpha$ be a composition of size $n$ and let $f \in \{ \rho,\psi,\omega \}$ be an automorphism $\QSym \to \QSym$. In this section, we identify the leading term of the dual immaculate expansion of $f(\mathfrak{S}^*_{\alpha})$ with respect to the lexicographic order. In other words, we identify the composition $\beta^f_\alpha \vDash n$, such that $[\mathfrak{S}^*_{\beta^f_{\alpha}}]f(\mathfrak{S}^*_{\alpha}) > 0$ and $[\mathfrak{S}^*_{\beta}]f(\mathfrak{S}^*_{\alpha})=0$ for all $\beta >_{l} \beta^f_{\alpha}$, \svw{namely, $d_{\beta^f_{\alpha}} >0$ in}
\[
f(\mathfrak{S}^*_{\alpha}) = d_{\beta^f_{\alpha}} \dI^*_{\beta^f_{\alpha}} + \sum_{\substack{\beta \vDash n \\ \beta <_\ell \beta^f_{\alpha}}} d_{\beta}\mathfrak{S}^*_{\beta}.
\]

The following lemma states that $\beta^f_\alpha$ is also the index of the leading term in the monomial \svw{quasisymmetric} and fundamental expansions of $f(\dI_\alpha^*)$, hence reducing this to the question of \svw{semi-standard} and standard immaculate tableaux by Definition \ref{defn:dI_over_M} and Proposition \ref{prop:dI_over_F}.

\begin{lemma}\label{lem:leading_terms_of_F_M_dI}
Let \svw{$G \in \QSym$} be homogeneous of degree $n$. Then the index of the leading term (with respect to the lexicographic order) of the $\mathfrak{S}^*$-expansion of $G$ is equal to the index of the leading term of the $M$-expansion of $G$, as well as the $F$-expansion of $G$. As well, the coefficient of the leading term in the $\mathfrak{S}^*$-expansion of $G$, is equal to that of the $M$- and $F$-expansions.
\end{lemma}

\begin{proof}
The proof follows since the transition matrices between $M_\alpha$ and $\dI^*_\alpha$ and between $M_\alpha$ and $F_\alpha$ are upper-triangular with ones on the diagonal.
\end{proof}

\begin{example}
Note the leading term of the quasisymmetric function $F_{(1,3)}$ is indexed by $(1,3)$ in the monomial \svw{quasisymmetric,} fundamental, and dual immaculate expansions.
    \[
        F_{(1,3)}= \mathfrak{S}^*_{(1,3)} = M_{(1, 3)} + M_{(1, 2, 1)} + M_{(1, 1, 2)} +M_{(1, 1, 1, 1)}.
    \]
\end{example}

\svw{We expand the definition of $\beta^f_\alpha$ to include $f = \Id$, where $\Id$ is the identity map.} 
To identify each $\beta^f_{\alpha}$, we will construct a  \svw{semi-standard tableau} $T^{f,M}_{\alpha}$ whose content is exactly $\beta^f_{\alpha}$. Each $T^{f,M}_{\alpha}$ will be a variant of an immaculate \svw{tableau,} corresponding to the involution $f$. In other words, for each $f \in \{\Id, \rho, \psi, \omega \}$, $T^{f,M}_{\alpha}$ is the unique tableau such that 
\begin{align*}
T^{\Id ,M}_{\alpha} &\in \SSIT(\alpha,\beta^{\Id}_{\alpha}),\\
  T^{\rho ,M}_{\alpha} &\in \SSRIT(\alpha^r,\beta^{\rho}_{\alpha}),\\
  T^{\psi ,M}_{\alpha} &\in \RSSIT(\alpha,\beta^\psi_{\alpha}), \\ T^{\omega ,M}_{\alpha} &\in \RSSRIT(\alpha^r,\beta^\omega_{\alpha}).  
\end{align*}
One can verify that $T^{f,M}_\alpha$ is indeed unique using Lemma \ref{lem:leading_terms_of_F_M_dI} and Definition \ref{defn:dI_over_M}.

\begin{remark}
    \svw{It} is clear that $\beta^\Id_\alpha = \alpha$. Furthermore, from the proof of \cite[Proposition 3.15]{bergLiftSchurHallLittlewood2014}, $T^{\Id,M}_\alpha$ is such that row $i$ is filled with entries $i$.
\end{remark}

\begin{example}\label{ex:leading_term_Mtab_Id}
The unique semi-standard \svw{tableau} of shape $\alpha=(1,4,2,2)$ and content $\beta^\Id_\alpha = \alpha$, i.e. $T^{\Id,M}_{(1,4,2,2)}$.
\begin{center}
  \begin{ytableau}
    4 & 4 \\
    3 & 3\\
    2 & 2 & 2 & 2 \\ 
    1
\end{ytableau}
\end{center}
\vspace{0.1cm}
\end{example}

\subsection{Leading term for $f=\rho$}

\begin{proposition}\label{prop:rho_leading_term}
    Let $\alpha \vDash n$ have length $\ell$. Then the leading term in the $\mathfrak{S}^*$-expansion of $\rho(\mathfrak{S}^*_{\alpha})$ is indexed by $\beta^\rho_{\alpha}=(n-\ell+1,1^{\ell-1})$, and $[\mathfrak{S}^*_{\beta^\rho_{\alpha}}]\rho(\mathfrak{S}^*_{\alpha})=1$. As well, the unique semi-standard reverse \svw{tableau} of shape $\alpha^r$ with content $\beta^\rho_{\alpha}$ is $T^{\rho,M}_\alpha$, and is such that column 1 contains $\{ 1,2,\dots,\ell \}$ and every other cell is filled with ones. 
\end{proposition}

\begin{example}\label{ex:leading_term_Mtab_rho}
The unique semi-standard reverse \svw{tableau} of shape $\alpha^r=(2,2,4,1)$ and content $\beta^\rho_\alpha = (6,1,1,1)$, i.e. $T^{\rho,M}_{(1,4,2,2)}$.
\begin{center}
       \begin{ytableau}
        4 \\
        3 & 1 & 1 & 1 \\ 
        2 & 1\\
        1 & 1 
     \end{ytableau}
\end{center}
\vspace{0.1cm}
\end{example}

\begin{proof}
By Proposition \ref{prop:variants_dI_as_images}, we have that 
\[
\rho(\mathfrak{S}^*_{\alpha})=\BdI^*_{\alpha^r}
\]
with $[M_{\beta}]\BdI^*_{\alpha^r}=|\SSRIT(\alpha^r,\beta)|$ for $\beta \vDash n$ by Definition \ref{defn:dI_over_M}. By Lemma \ref{lem:leading_terms_of_F_M_dI}, the leading term with respect to the dual immaculate expansion is equal to the leading term with respect to the monomial \svw{quasisymmetric} expansion, with equal coefficient.

Let $\beta^\rho_{\alpha}=(n-\ell+1,1^{\ell-1})$. 

Suppose $T \in \SSRIT(\alpha^r,\beta)$ with $\beta \geq_{l} \beta^\rho_{\alpha}$. For each entry $m \geq 1$, we will consider how many cells can contain $m$, versus how many copies of entry $m$ the tableau $T$ must have, as specified by the content $\beta$.

Since $\beta_{1} \geq (\beta^\rho_{\alpha})_1 = n-\ell+1$,  $T$ must have at least $n-\ell+1$ copies of entry $1$. On the other hand, since the first column of $T$ is strictly increasing from bottom to top, there is at most one copy of 1 in the first column, in cell $(1,1)$. There are $n-\ell$ cells outside of the first column. Thus, all these cells must have entry 1. In particular, if $\beta_{1}>n-\ell+1$, such a tableau does not exist, and so $\beta_{1}=n-\ell+1$. As well, $1 = T(1,1)$ and $1 = T(i,j)$ for all $i$ and all $j \geq 2$. 
\svw{From here,} if $\beta >_{l} \beta^\rho_{\alpha}$, then $\beta_{m}\geq 2$ for some $m \geq 2$, and so $T$ has at least 2 copies of entry $m$. However, the only remaining empty cells are in column 1 so $T$ has at least two copies of entry $m$ in the first column. But this contradicts that the first column is strictly increasing. Thus, it must be that $\beta=\beta^\rho_{\alpha}$. In this case, there is a unique way to distribute the entries $\geq 2$, since they must be strictly increasing from bottom to top within the first column. i.e. place entry $m$ into cell $(m,1)$ for $m \in \{2,\ldots, \ell \}$. 

This shows that $[M_{\beta^\rho_{\alpha}}]\rho(\mathfrak{S}^*_{\alpha})=1$ and  $[M_{\beta}]\rho(\mathfrak{S}^*_{\alpha})=0$ if $\beta >_{l} \beta^\rho_{\alpha}$. 
\end{proof}

\subsection{Leading term for $f=\psi$}

\begin{proposition}\label{prop:psi_leading_term}
Let $\alpha \vDash n$ have length $\ell$. Then the leading term in the $\mathfrak{S}^*$-expansion of $\psi(\mathfrak{S}^*_{\alpha})$ is indexed by $\beta^\psi_{\alpha}=(c_{i})_{i}$ where $c_{i}$ is the number of cells in the $i$-th column of the diagram of $\alpha$, and $[\mathfrak{S}^*_{\beta^\psi_{\alpha}}]\psi(\mathfrak{S}^*_{\alpha})=1$.
As well, the unique \svw{row-strict  semi-standard tableau} of shape $\alpha$ with content $\beta^\psi_{\alpha}$ is $T^{\psi,M}_\alpha$, and is such that column $i$ is filled  with entries $i$.
\end{proposition}

\begin{example}\label{ex:leading_term_Mtab_psi}
    The unique \svw{row-strict  semi-standard tableau} of shape $\alpha=(1,4,2,2)$ and content $\beta^\psi_\alpha = (4,3,1,1)$, i.e. $T^{\psi,M}_{(1,4,2,2)}$.
\begin{center}
 \begin{ytableau}
    1 & 2 \\
    1 & 2\\
    1 & 2 & 3 & 4 \\ 
    1 
\end{ytableau}
\end{center}
\vspace{0.1cm}
\end{example}

\begin{proof}
By Proposition \ref{prop:variants_dI_as_images}, we have that 
\[
\psi(\mathfrak{S}^*_{\alpha})=R\mathfrak{S}^*_{\alpha}
\]
with $[M_{\beta}]R\mathfrak{S}^*_{\alpha}=|\RSSIT(\alpha,\beta)|$ for $\beta \vDash n$ by Definition \ref{defn:dI_over_M}. By Lemma \ref{lem:leading_terms_of_F_M_dI}, the leading term with respect to the dual immaculate expansion is equal to the leading term with respect to the monomial \svw{quasisymmetric} expansion, with equal coefficient.

Let $\beta^{\psi}_{\alpha}=(c_{i})_{i}$ where $c_{i}$ is the number of cells in the $i$-th column of the diagram of $\alpha$.

Suppose $T \in \RSSIT(\alpha,\beta)$ with $\beta \geq_{l} \beta^\psi_{\alpha}$. For each entry $m \geq 1$, we will consider how many cells can contain $m$, versus how many copies of entry $m$ the tableau $T$ must have, as specified by the content $\beta$.

Since $\beta_{1} \geq c_{1}$, $T$ must contain at least $c_{1}$ copies of entry 1. Since each row is strictly increasing left to right, each row can contain at most one copy of entry 1, which must be in cell $(i,1)$ for some $i$. In particular, if $\beta_{1}>c_{1}$ this is not possible as the number of entries is strictly greater than the length of the first column. Thus, it must be that $\beta_{1}=c_{1}$. Moreover, the positions of the entries $1$ are uniquely determined, in particular, the first column has to be filled with ones, and no other cell can contain a $1$. 

By induction, we can assume each of the first $k$ columns consist exactly of $c_{k}$ copies of entry $k$ and $\beta_{i}=(\beta^\psi_{\alpha})_{i}$ for all $i \leq k$. Similarly, since the rows are \svw{strictly} increasing left to right, all copies of entry $k+1$ must be leftmost among the empty cells in each row, and there is at most one copy of entry $k+1$ per row, since the rows are strictly increasing. \svw{Thus, if $\beta \geq_{l}\beta^\psi_{\alpha}$ then  $\beta_{k+1}\geq(\beta^\psi_{\alpha})_{k+1}$ but $\beta_{k+1}>(\beta^\psi_{\alpha})_{k+1}$ is not possible,  and so column $k+1$ must consist of exactly $c_{k+1}$ copies of entry $k+1$.}

This shows that $[M_{\beta^\psi_{\alpha}}]\psi(\mathfrak{S}^*_{\alpha})=1$ and  $[M_{\beta}]\psi(\mathfrak{S}^*_{\alpha})=0$ if $\beta >_{l} \beta^\psi_{\alpha}$. 
\end{proof}

\subsection{Leading term for $f=\omega$}

\begin{proposition}\label{prop:ome_leading_term}
    Let $\alpha \vDash n$ have length $\ell$. Then the leading term in the $\mathfrak{S}^*$-expansion of $\omega(\mathfrak{S}^*_{\alpha})$ is indexed by $\beta^\omega_{\alpha}$ where 
    \[
    (\beta^\omega_{\alpha})_{m}=\#\{ i \mid \alpha_{i}-1\geq m \} + \#\{ i \mid \max_{j \leq i} \{ \alpha^r_{j} \} =m \},
    \]
    and $[\mathfrak{S}^*_{\beta^\omega_{\alpha}}]\omega(\mathfrak{S}^*_{\alpha})=1$. 
As well, the unique \svw{row-strict semi-standard  reverse tableau} of shape \svw{$\alpha ^r$} with content $\beta^\omega_{\alpha}$ is $T^{\omega,M}_\alpha$ and is such that 
\[
\Row_{i}(T^{\omega,M}_{\alpha})=(\max_{j \leq i} \{ \alpha^r_{j} \} ,\alpha^r_{i}-1,\dots,2,1).
\]
\end{proposition}

\begin{example}\label{ex:leading_term_Mtab_ome}
    The unique \svw{row-strict semi-standard  reverse tableau} of shape $\alpha^r=(2,2,4,1)$ and content $\beta^\omega_\alpha = (3,3,1,2)$, i.e. $T^{\omega,M}_{(1,4,2,2)}$.
    \begin{center}
       \begin{ytableau}
    4 \\
    4 & 3 & 2 & 1 \\ 
    2 & 1\\
    2 & 1 
\end{ytableau}
    \end{center}
    \vspace{0.1cm}
\end{example}

\begin{proof}
By Proposition \ref{prop:variants_dI_as_images}, we have that 
\[
\omega(\mathfrak{S}^*_{\alpha})=R\BdI^*_{\alpha^r}
\]
with $[M_{\beta}]R\BdI^*_{\alpha^r}=|\RSSRIT(\alpha^r,\beta)|$ for $\beta \vDash n$ by Definition \ref{defn:dI_over_M}. By Lemma \ref{lem:leading_terms_of_F_M_dI}, the leading term with respect to the dual immaculate expansion is equal to the leading term with respect to the monomial \svw{quasisymmetric} expansion, with equal coefficient.

For $m \geq 1$, let $(\beta^\omega_{\alpha})_{m}=\#\{ i \mid \alpha_{i}-1\geq m \} + \#\{ i \mid \max_{j \leq i} \{ \alpha^r_{j}\} =m \}$. 

Suppose $T \in \RSSRIT(\alpha^r,\beta)$ with $\beta \geq_{l} \beta^\omega_{\alpha}$. For each entry $m \geq 1$, we will consider how many cells can \ME{contain} $m$, versus how many copies of entry $m$ the tableau $T$ must have, as specified by the content $\beta$.

Observe that since the rows are strictly increasing right to left, 
\begin{align}\label{eq:first_cond_ome_leading_term}
T(i,j) \geq \alpha^r_{i}-(j-1), \hspace{0.2cm} \text{for all } i,j.
\end{align}
In particular, $T(i,1) \geq \alpha^r_{i}$ for all $i$. Moreover, since the first column is weakly increasing from bottom to top, $T(i,1)\geq T(j,1)\geq\alpha^r_{j}$ for all $i>j$ and so
\begin{align}\label{eq:second_cond_ome_leading_term}
T(i,1)\geq \max_{j\leq i} \{ \alpha^r_{j}\} \hspace{0.2cm} \text{for all } i.
\end{align}
As well, by the \svw{row-strict}  increasing condition, each row can have at most one copy of each entry.

We consider how entry 1 can fill $T$. Since $\beta_1 \geq (\beta^\omega_\alpha)_1$, $T$ has at least $\svw{(\beta^\omega_\alpha)_1=}\#\{ i \mid \alpha_{i}-1\geq 1 \} + \#\{ i \mid \max_{j \leq i} \{ \alpha^r_{j} \} =1 \}$ copies of entry 1.
For $i$ with $\alpha^r_{i} \geq 2$, we can always have $1 = T(i,\alpha^r_{i})$ without affecting the first column. This contributes at most $\# \{ i \mid \alpha_{i}-1 \geq 1\}$ copies of the entry 1 to the content of $T$. 
For $i$ with $\alpha^r_{i}=1$, we need to make sure the first column is weakly increasing from bottom to top. 

By the necessary \svw{row and column conditions} identified above, 
$$
\{ i \mid T(i,1)=1 \} \subseteq \{ i \mid \max_{ j \leq i} \{ \alpha_{j}^r \} \leq 1 \}=\{ i \mid \max_{ j \leq i} \{ \alpha_{j}^r \} = 1 \}.
$$
Thus, there are at most $(\beta^\omega_\alpha)_1$ cells which could contain entry 1. In particular, if $\beta_1 > (\beta^\omega_\alpha)_1$ then such a tableau does not exist. It follows that $\beta_1 = (\beta^\omega_\alpha)_1$ and the positions of the entries 1 in $T$ are uniquely determined. In particular, $1 = T(i,\alpha^r_i)$ if $\alpha^r_i \geq 2$ and $1 = T(i,1)$ if \svw{$\max_{j \leq i} \{ \alpha^r_j \} =1$.} 

By induction, assume for $p \leq m$, $\beta_{p}=(\beta^\omega_{\alpha})_{p}$, and the \svw{entries} $p \leq m$ are assigned exactly to the cells $(i,\alpha^r_{i}-(p-1))$ for $i$ such that $\alpha^r_{i}\geq p+1$ and $(i,1)$ for $i$ such that $\max_{j \leq i} \{ \alpha^r_{j} \} =p$.

We now consider how the entry $m+1$ can fill $T$. Since $\beta_{m+1} \geq (\beta^\omega_\alpha)_{m+1}$, $T$ has at least $(\beta^\omega_{\alpha})_{m+1}=\#\{ i \mid \alpha_{i}-1\geq m+1 \} + \#\{ i \mid \max_{j \leq i} \{ \alpha^r_{j} \} =m+1 \}$ copies of entry $m+1$. 
For $i$ such that $\alpha^r_{i}\geq m+2$, there exists an empty cell in row $i$ that is not in column 1, and we can assign $m+1$ rightmost among the empty cells in this row, in particular, to cell $(i,\alpha^r_{i}-m)$. 
For $i$ with $\alpha^r_{i}\leq m+1$, we can fill cell $(i,1)$ with entry $m+1$ if this does not contradict that the first column is weakly increasing from bottom to top. By the earlier necessary \svw{row and column conditions,}
\[
\{ i \mid T(i,1)=m+1 \} \subseteq \{ i \mid \max_{ j \leq i} \{ \alpha_{j}^r \} \leq m+1 \}.
\]
Suppose for the sake of a contradiction that $i$ is such that $T(i,1)=m+1$ but $\max_{j \leq i} \{ \alpha^r_{j} \} < m+1$. By our inductive hypothesis, the cell $(i,1)$ has already been filled with \svw{an entry} $\leq m$. Thus, 
\[
\{ i \mid T(i,1)=m+1 \} \subseteq \{ i \mid \max_{ j \leq i} \{ \alpha_{j}^r \} = m+1 \}.
\]
It follows that 
\begin{align*}
    \{ (i,j) \mid T(i,j)=m+1 \} \subseteq \{ (i,\alpha^r_{i}-m) &\mid \alpha^r_{i}-1 \geq m+1 \} \\
    &\sqcup \{ (i,1) \mid \max_{j \leq i } \{ \alpha^r_{j} \} =m+1 \}.
\end{align*}
In particular, there are at most $(\beta^\omega_\alpha)_{m+1}$ cells which could contain entry $m+1$. If $\beta_{m
+1}>(\beta^\omega_\alpha)_{m+1}$ \svw{then} such a tableau does not exist. It follows that $\beta_{m+1} =(\beta^\omega_\alpha)_{m+1}$ and the positions of the entries $m+1$ are uniquely determined. In particular, $m+1$ is assigned to the cells $(i,\alpha^r_{i}-m)$ if $\alpha^r_{i}-1\geq m+1$ and $(i,1)$ if $\max_{j \leq i} \{ \alpha^r_{j} \}=m+1$.

Note the proof constructs $T^{\omega,M}_{\alpha}$ such that for all $i$,
\[
\Row_{i}(T^{\omega,M}_{\alpha})=(\max_{j \leq i} \{ \alpha^r_{j} \},\alpha^r_{i}-1,\dots,2,1).
\]

This shows that $[M_{\beta^\omega_{\alpha}}] \omega(\mathfrak{S}^*_{\alpha})=1$ and  $[M_{\beta}]\omega(\mathfrak{S}^*_{\alpha})=0$ if $\beta >_{l} \beta^\omega_{\alpha}$.
\end{proof}

\begin{remark}
    Equivalently, $T^{\omega,M}_{\alpha}$ is the unique \svw{tableau} that meets the necessary conditions (\ref{eq:first_cond_ome_leading_term}) and (\ref{eq:second_cond_ome_leading_term}) with equality. 
\end{remark}

Next, we prove some facts about $\beta^\omega_\alpha$ which will be useful to us in the proof of the main theorem in Section \ref{ch:dI_images}. 

\begin{lemma}\label{lemma:explicit_beta_if_alpha_divboard}
Let $\alpha \vDash n$ be a diving-board with length $\ell$. Assume $\alpha \neq (1^n)$ and the long row has length \ME{$k \geq 2$.} Then $\ell(\beta^\omega_{\alpha})=k$. As well, $\beta^\omega_{\alpha}$ is a diving-board if and only if $\alpha$ is a bottom- or top-aligned hook.
\end{lemma}

\begin{example}
    \svw{We calculate} $T^{\omega, M}_\alpha$ and $\beta^\omega_\alpha$ when $\alpha$ is a \svw{bottom-aligned} hook, a diving-board that is not a top- or bottom-aligned hook, as well as a composition that is not a diving-board. From left to right, \svw{$T^{\omega,M}_{(4,1,1)}$,} $T^{\omega,M}_{(1,4,1)}$, \svw{$T^{\omega,M}_{(2,3,1)}$.} 
    \begin{center}
         \begin{ytableau}
        4 & 3 & 2 & 1 \\ 
        1 \\
        1
    \end{ytableau} \hspace{0.8cm}
    \begin{ytableau}
        4 \\ 
        4 & 3 & 2 & 1 \\
        1 
    \end{ytableau} \hspace{0.8cm}
    \begin{ytableau}
        3 & 1 \\ 
        3 & 2 & 1 \\
        1
    \end{ytableau}
    \end{center}
    \vspace{0.1cm}
    Note that \svw{$\beta^\omega_{(4,1,1)} = (3,1,1,1), \, 
        \beta^\omega_{(1,4,1)} = (2,1,1,2), \,
        \beta^\omega_{(2,3,1)} = (3,1,2)$.}
\end{example}

\begin{proof}
Assume $\alpha$ is a \svw{diving-board} such that the long row with length \ME{$k \geq 2$} occurs at index $1 \leq i \leq \ell$. 

Given $\alpha$, we compute the explicit form of $\beta_{\alpha}^\omega$ using the fact that $\beta_{\alpha}^\omega$ is the content of $T_{\alpha}^{\omega,M}$, from Proposition \ref{prop:ome_leading_term}. 
Recall, the tableau $T^{\omega,M}_{\alpha}$ is a \svw{row-strict} semi-standard reverse immaculate \svw{tableau} with shape $\alpha^r$ and is defined by 
\[
\Row_{p}(T^{\omega,M}_{\alpha})=(\max_{j \leq p} \{ \alpha^r_{j} \},\alpha^r_{p}-1,\dots,2,1)
\]
for each $1 \leq p \leq \ell$.

Note that $\alpha^r$ is a diving-board such that the long row of length $k$ occurs at index $\ell-i+1$. 

Then, for $1 \leq p < \ell-i+1$, $\alpha^r_{p}=1$, and so $\Row_{p}(T^{\omega,M}_{\alpha})=(1)$. 
As well, $\Row_{\ell-i+1}(T^{\omega,M}_{\alpha})=(k,k-1,\dots,2,1)$.
Finally, for $\ell-i+1 < p \leq \ell$, $\Row_{p}(T^{\omega,M}_{\alpha})=(k)$. 

Thus, $\beta^{\omega}_{\alpha} = \content(T^{\omega,M}_{\alpha})=(\ell-i+1, 1^{k-2}, i)$ and in particular $\ell(\beta^\omega_{\alpha})=k$.
Note that $\beta^\omega_{\alpha}$ is a diving-board if and only if $\ell-i+1=1$ or $i=1$. However, this happens if and only if $i=1$ or $i=\ell$, ie. exactly when $\alpha$ is a bottom- or top-aligned hook.
\end{proof}

\begin{lemma}\label{lemma:characterize_alpha_divboard_using_beta}
    Let $\alpha \vDash n$ with $\alpha \neq (1^n)$. Then $\alpha$ is a diving-board if and only if $\ell(\alpha)+\ell(\beta^\omega_{\alpha})>n$. In particular, if $\alpha$ is a diving-board then $\ell(\alpha)+\ell(\beta^\omega_{\alpha})=n+1$. 
\end{lemma}

\begin{proof}
Denote $\beta = \beta^\omega_{\alpha}$. 
    Suppose $\alpha$ is a diving-board, where the long row of length \ME{$k \geq 2$} occurs at index $1 \leq i \leq \ell(\alpha)$. By Lemma \ref{lemma:explicit_beta_if_alpha_divboard}, $\ell(\beta)=k$. Thus, $\ell(\alpha)+\ell(\beta) = \ell(\alpha)+k = n+1$. 

Suppose on the other hand that $\ell(\alpha)+\ell(\beta)>n$. Recall that $\beta = \content(T^{\omega,M}_{\alpha})$ as in Proposition \ref{prop:ome_leading_term}. Thus, $\{ 1,2,\dots,\ell(\beta) \}$ is the set of distinct entries in $T^{\omega,M}_{\alpha}$. Recall, by Proposition \ref{prop:ome_leading_term}, for each $1 \leq i \leq \ell(\alpha)$, 
\[
\Row_{i}(T^{\omega,M}_{\alpha})=(\max_{j \leq i} \{ \alpha^r_{j}\},\alpha^r-1,\dots,2,1).
\]
Since $\ell(\beta)>n-\ell(\alpha)$, and $\ell(\beta)$ is an entry of $T$ in column 1, \svw{because rows increase from \ME{right to left}} by construction 
there exists some $i$ such that $\max_{j \leq i} \{ \alpha^r_{j} \} > n-\ell(\alpha)$. Hence there exists some $i$ such that $\alpha_{i}>n-\ell(\alpha)$. Since the first column of $\alpha$ consists of $\ell(\alpha)$ cells, and $\alpha$ has $n$ cells, it must be that $\alpha$ is a diving-board with $\alpha_{i}=n-\ell(\alpha)+1$. 
\end{proof}

\section{Classification of dual immaculate functions in the fundamental basis}\label{ch:dI_in_F_classification}
Next, we will consider the question of when the $F$-expansion of $\mathfrak{S}^*_{\alpha}$ is `simplest'. In other words, we will characterize $\alpha$ such that the expansion consists of exactly one term. This result will not only be useful in later sections, but also the main result of this work considers the analogue of this question for the $\dI^*$-expansion of variants of dual immaculate functions.

We begin with an explicit computation of the dual immaculate functions indexed by \svw{diving-boards.} 

\begin{proposition}\label{prop:diving_board_explicit_comp}
Let $\alpha \vDash n$ be \ME{$(1^n)$ or a }\svw{diving-board} with the long row of length $k$ at index $i$, i.e. $\alpha = (1^{i-1},k,1^{\ell-i})$ \ME{for $k \geq 1$}. Then
\[
\mathfrak{S}^*_{\alpha} = \sum_{S} F_{\Comp([n-1] \setminus (S-1))}
\]
where the sum is over all $S \subseteq \{ i+1,\dots,n \}$ with $|S| = k-1$. In particular, there are ${n-i \choose k-1}$ terms in the $F$-expansion of $\mathfrak{S}^*_{\alpha}$. 
\end{proposition}

\begin{example}
    We have $\mathfrak{S}^*_{(3,1,1)}= F_{(3,1,1)} + F_{(2,2,1)} + F_{(2,1,2)} + F_{(1,3,1)} + F_{(1,2,2)} + F_{(1,1,3)} $, from the following standard immaculate tableaux. As in Proposition \ref{prop:diving_board_explicit_comp}, $n=5$, $k=3$, $i=1$, and ${n-i \choose k-1} = {4 \choose 2} = 6$.
    \center{
            {\small
    \begin{ytableau}
        5 \\
        4 \\ 
        1 & 2 & 3
    \end{ytableau}} \hspace{0.2 cm}
    {\small
    \begin{ytableau}
        5 \\
        3 \\ 
        1 & 2 & 4
    \end{ytableau}} \hspace{0.2cm}
    \begin{ytableau}
        4 \\
        3 \\ 
        1 & 2 & 5 
    \end{ytableau} \hspace{0.2cm}
    {\small \begin{ytableau}
        5 \\
        2 \\ 
        1 & 3 & 4
    \end{ytableau}} \hspace{0.2cm}
    {\small \begin{ytableau}
        4 \\
        2 \\ 
        1 & 3 & 5
    \end{ytableau}} \hspace{0.2cm}
    {\small \begin{ytableau}
        3 \\
        2 \\ 
        1 & 4 & 5
    \end{ytableau} }
    }
    \vspace{0.1cm}
\end{example}

\begin{proof}
    Recall that by Proposition \ref{prop:dI_over_F},
    \[
    \mathfrak{S}^*_\alpha = \sum_{T \in \SIT(\alpha)} F_{\Des_{\mathfrak{S}^*}(T)},
    \] where the sum runs over all standard immaculate \svw{tableaux} $T$ of shape $\alpha$. Using the correspondence between compositions of $n$ and subsets of $[n-1]$, we can assume each index $\Des_{\mathfrak{S}^*}(T)$ is a subset of $[n-1]$. \svwrev{The result for $k=1$ is a straightforward calculation.}

    To prove the result \svwrev{for $k\geq 2$,} we investigate which \ME{tableaux} belong to $\SIT(\alpha)$.

    For each subset $S \subseteq \{ i+1,\dots,n \}$, $|S| = k-1$, define $T_S$ to be the tableau of shape $\alpha$ such that the entries from $S$ fill cells $(i,2),(i,3),\ldots,(i,k)$,
    such that the row is increasing left to right, while the entries in $[n] \setminus S$ fill column 1, such that it is increasing from bottom to top.

    \begin{example}
       \ME{Let $\alpha=(1,4,1) \vDash 6$. Then $T_S$ of shape $\alpha$ for $S=\{3,5,6 \}$ is as follows.}      \begin{center}
       \begin{ytableau}
        4 \\
        2 & 3 & 5 & 6 \\ 
        1 \\
    \end{ytableau}
        \end{center}\svw{Note the $\mathfrak{S}^*$-descent} set is $\{1,3 \}$.
        \vspace{0.1cm}
    \end{example}

    We show that $\SIT(\alpha) = \{T_S \mid S \subseteq \{ i+1,\dots,n \},\, |S| = k-1 \}$.
    
    Let $T \in \SIT(\alpha)$ and let $S$ denote the set of entries \svw{of $T$} in cells $(i,2),(i,3),\ldots,(i,k)$. Then, $|S|=k-1$. Since the first column is increasing from bottom to top, $T(i,1) \geq i$. Since row $i$ is increasing left to right, $S\subseteq \{ i+1,\dots,n \}$. Thus, $\svw{T} \in \{T_S \mid S \subseteq \{ i+1,\dots,n \},\, |S| = k-1 \}$.

    On the other hand, consider $T_S$ for $S \subseteq \{ i+1,\dots,n \},\, |S| = k-1$. To see that $T_S$ is in fact an immaculate tableau, it remains to verify $T_S(i,1) \leq T_S(i,2)$. 
    Since the entry $T_S(i,1)$ is in $[n] \setminus S$, we have that $T_S(i,1) \leq i$. Since $T_S(i,2) \in S$, we have that $T_S(i,2) \geq i+1$ and so the $i$-th row of $T_S$ is increasing left to right. It follows by definition that $T_S \in \SIT(\alpha)$. 

    Thus,
    \[
    \mathfrak{S}^*_\alpha = \sum_S F_{\Des_{\mathfrak{S}^*}(T_S)}
    \]
    where the sum runs over all $S \subseteq \{ i+1,\dots,n \},\, |S| = k-1$.

    It remains to show that $\Des_{\mathfrak{S}^*}(T_S)=[n-1]\setminus (S-1)$. In other words, for $m \in [n-1]$, $m \in \Des_{\mathfrak{S}^*} (T_{S})$ if and only if $m+1 \not\in S$.

    Suppose that $m+1 \in S$. Since $S \subseteq \{ i+1,\dots,n \}$, we have $m \geq i$. 
Suppose for the sake of a contradiction that the row index of $m$ is $R_m < i$. Then, there are at least $n-(i-1)$ cells above and to the right of $m$, and each must have an entry that is strictly greater than $m$ since the rows and columns are increasing. However, there are only $n-m < n-(i-1)$ entries strictly greater than $m$, and so \svw{we have a contradiction and hence} $R_m \geq i$. But, $R_{m+1}=i$ by definition of $T_S$, since $m+1 \in S$. Thus, $R_m \geq i = R_{m+1}$ and so $m \not\in \Des_{\mathfrak{S}^*} (T_{S})$.

On the other hand, suppose $m+1 \not\in S$.

If $R_m \neq i$ then $m$ is contained in the first column of \svw{$T_S$.} Since $m+1 \notin S$, $m+1$ is also in the first column of \svw{$T_S$.} Since the first column is increasing from bottom to top, $R_m < R_{m+1}$ and so $m \in \Des_{\mathfrak{S}^*} (T_S)$.

Suppose $R_m = i$ and $R_{m+1} \leq i$ for the sake of a contradiction. Then, the entry $m+1$ is in the first column of \svw{$T_S$ because $m+1\not\in S$.} Since the first column is increasing from bottom to top, and the $i$-th row is increasing from left to right, $m+1 \leq \svw{T_S(i,1)} \leq m$, a contradiction \svw{so $R_{m+1}>i$.} Thus, $R_m < R_{m+1}$ and so $m \in \Des_{\mathfrak{S}^*} (T_S)$.
\end{proof}

Next, we answer the question of when the fundamental expansion of $\mathfrak{S}^*_{\alpha}$ consists of exactly one term.

\begin{proposition}\label{prop:ex_1_descent}
Let $\alpha \vDash n$.
    Then $\mathfrak{S}^*_{\alpha} = F_{\beta}$ \ME{\svw{for some $\beta$,}} if and only if $\alpha$ is a top-aligned hook. In this case, $\mathfrak{S}^*_{\alpha} = F_{\alpha}$.
\end{proposition}

\begin{proof}
Denote $\ell = \ell(\alpha)$. 
We first prove the result when $\alpha$ is a diving-board, using Proposition \ref{prop:diving_board_explicit_comp}.
Suppose $\alpha$ is a diving-board where the long row of length $k$ occurs at index $i$, for some $1 \leq i \leq \ell$. Then, there are ${n-i \choose k-1}$ terms in the $F$-expansion of $\mathfrak{S}^*_{\alpha}$. 

Suppose $\alpha$ is a top-aligned hook, i.e. $i=\ell$. Then, $n-i = n-\ell = k-1$, and so ${n-i \choose k-1} = {k-1 \choose k-1} =1$. Thus, the indexing set indeed consists of exactly one element, and so $\mathfrak{S}^*_\alpha = F_\beta$ in this case. 

On the other hand, suppose $\alpha$ is not a top-aligned hook, i.e. $i < \ell$. Then, $n-i > n-\ell = k-1$, and so ${n-i \choose k-1} \geq 2$. Thus, $\mathfrak{S}^*_{\alpha} \neq F_{\beta}$ for all $\beta \vDash n$.  

Now, suppose $\alpha$ is not a diving-board. Then, there exists $i < j$ such that $\alpha_{i}, \alpha_{j} \geq 2$.  We construct two standard immaculate tableaux $T,T'$ with distinct $\mathfrak{S}^*$-descent sets. Let $T \in \SIT(\alpha)$ be such that $n = T(j, \alpha_j)$ and $n-1 = T(i,\alpha_i)$. Note such a tableau exists since $n-1,n$ are the maximal entries of $T$, are rightmost within their rows, and are not in the first column. Notice that $n-1 \in \Des_{\mathfrak{S}^*}(T)$ since $i<j$. 

Let $T' \in \SIT(\alpha)$ be obtained from $T$ by switching the positions of entries $n-1$ and $n$. The increasing conditions on the rows and the first column will still be met since $n,n-1$ are the largest entries and are not contained in the first column. Notice that $n-1 \notin \Des_{\mathfrak{S}^*}(T')$ since $i<j$. Thus, $\Des_{\mathfrak{S}^*}(T) \neq \Des_{\mathfrak{S}^*}(T')$ and so $\mathfrak{S}^*_{\alpha}$ has at least two terms in \svw{its} fundamental expansion. 

If $\dI^*_\alpha = F_\beta$ then by Lemma \ref{lem:leading_terms_of_F_M_dI}, $\alpha = \beta$.
\end{proof}

\begin{example}
As in the proof of Proposition \ref{prop:ex_1_descent}, examples of fillings $T$ and $T'$ of $\alpha=(1,3,2,1)$ with distinct descent sets. The $\mathfrak{S}^*$-descent sets of $T,T'$ are $\{1,3,4,6 \}$ and $\{1,3,4 \}$, respectively.
\begin{center}
           $T=$\begin{ytableau}
        5 \\
        4 & 7 \\
        2 & 3 & 6 \\ 
        1 
        \end{ytableau} \hspace{0.5cm}
        $T'=$\begin{ytableau}
        5 \\
        4 & 6 \\
        2 & 3 & 7 \\ 
        1 
        \end{ytableau}
\end{center}
\vspace{0.1cm}
\end{example}

\section{Necessary and sufficient conditions for the existence of immaculate tableaux}\label{ch:nec_suff_conds}
Throughout this section, let $n$ be a positive integer, and let $\alpha,\beta$ denote two compositions of size $n$. Furthermore, let $T$ be a standard immaculate tableau.

We identify sufficient and necessary conditions to have $\SIT_{\dI^*}(\alpha ; \beta) \neq \emptyset$, or equivalently, $[F_\beta] \dI^*_\alpha \neq 0$. Compare \svw{these conditions} to Proposition \ref{prop:prop315} which is a necessary condition for $\SSIT(\alpha, \beta) \neq \emptyset$, or equivalently, $[M_\beta] \dI^*_\alpha \neq 0$. 

The conditions will follow by analyzing the restrictions placed on a standard immaculate tableau by specifying the $\mathfrak{S}^*$-descent composition or the $R\mathfrak{S}^*$-descent composition, and using that the $R\mathfrak{S}^*$-descent composition of a tableau is the complement of the $\mathfrak{S}^*$-descent composition. The ability to think of a standard immaculate \svw{tableau} as a row-strict immaculate \svw{tableau} is useful to us as it picks up a different `architecture' of the filling.

\subsection{Necessary conditions for $\SIT_{\mathfrak{S}^*}(\alpha;\beta) \neq \emptyset$}

\subsubsection{Derived from \svw{the} $\mathfrak{S}^*$-descent composition}

By Remark \ref{rem:descent_comp_as_runs}, the $\mathfrak{S}^*$-descent composition records how the entries of a standard immaculate tableau, from $1$ to $n$, can be read weakly down the rows, forming maximal runs of integers. This observation will give a necessary condition on the row indices of the \svw{entries} of $T$. By considering the number of cells which contain the first $i$ runs of $T$, we derive a necessary condition in terms of the dominance order of $\alpha$ and $\beta$. 

\begin{lemma}\label{lemma:row_indices_runs}
Suppose $\alpha,\beta \vDash n$ and $T \in \SIT_{\mathfrak{S}^*}(\alpha;\beta)$. Let $m$ be an entry of $T$ in $\{1,2,\ldots, \beta_1 +\cdots + \beta_k\}$. Then the row index of $m$ is $R_m(T) \leq k$. 
\end{lemma}

\begin{proof}
We proceed by induction on $1 \leq \tilde{k} \leq \ell(\alpha)$.
Since each row is increasing left to right, the entry 1 is contained in the first column of $T$. Since the first column is increasing from bottom to top, $1 = T(1,1)$\ME{.} Since $\Des_{\mathfrak{S}^*}(T) = \beta$, the entries $\{2,\ldots,\beta_1 \}$ have to be weakly below 1. Thus, the entries $\{1,\dots,\beta_{1} \}$ are contained in row 1 \ME{so the result holds for $\tilde{k}=1$}. 
Suppose all the \svw{entries} $\{ 1,2,\svw{\ldots, }  \beta_{1}+\dots+\beta_{\tilde{k}} \}$ are contained in rows $\{1,2,\ldots,\tilde{k} \}$.
Let $\tilde{m}=\beta_{1}+\dots+\beta_{\tilde{k}}+1$. Suppose for the sake of a contradiction that $R_{\tilde{m}}=j > \tilde{k}+1$. Let $p = T(\tilde{k}+1,1)$. Then, since the first column is increasing from bottom to top, and the $j$-th row is increasing from left to right,
\[
p=T(\tilde{k}+1,1) < T(j,1) \leq \tilde{m}.
\]
Then, $p$ is an entry of $T$ strictly less than $\tilde{m}$, but all entries strictly less than $\tilde{m}$ are in rows $\{ 1,\ldots,\tilde{k} \}$ by assumption, while $R_{p}=\tilde{k}+1$, a contradiction. Thus, $R_{\tilde{m}}=\tilde{k}+1$. Since $\Des_{\mathfrak{S}^*}(T)=\ME{\beta}$, the row indices \ME{of entries $\beta_{1}+\cdots+\beta_{\tilde{k}}+1, \ldots \svwrev{, }\beta_{1}+\cdots+\beta_{\tilde{k}+1}$} must be weakly decreasing, i.e. $\svw{\tilde{k}+1} = R_{\beta_{1}+\cdots+\beta_{\tilde{k}}+1} \geq \cdots \geq R_{\beta_{1}+\cdots+\beta_{\tilde{k}+1}}$.
\end{proof}

\begin{example}\label{ex:tab_ex_for_nec_conds}
    \svw{A} standard immaculate tableau $T$ of shape $(4,5,1)$ with $\mathfrak{S}^*$-descent composition $\beta = (2,2,2,3,1)$. Note that entries $\{1,2 \}$ are contained in row 1, while the entries $\{3,4\}$ are contained in rows 1,2.
    \begin{center}
\begin{ytableau}
    7 \\
    3 & 5 & 6 & 8 & 10 \\ 
    1 & 2 & 4 & 9 
\end{ytableau} 
    \end{center}
\vspace{0.1cm}
\end{example}

\begin{proposition}\label{prop:nec_cond_L_dom}
Let $\alpha,\beta \vDash n$ \ME{be} such that $\SIT_{\mathfrak{S}^*}(\alpha;\beta) \neq \emptyset$. Then $\beta \leq_{d} \alpha$. 
\end{proposition}

\begin{proof}
Let $1 \leq k \leq \min \{ \ell(\alpha),\ell(\beta) \}$. 
Then, $\sum_{i=1}^k \alpha_{i}$ is the number of cells in the first $k$ rows of the diagram of $\alpha$, while $\sum_{i=1}^k \beta_{i}$ is the number of entries in $\{1,2,\ldots, \beta_{1}+\dots+\beta_{k} \}$. Since the set of these entries is contained in the first $k$ rows of $\alpha$ by Lemma \ref{lemma:row_indices_runs}, it follows that $\sum_{i=1}^k \beta_{i} \leq \sum_{i=1}^k \alpha_{i}$.
\end{proof}

\svw{As a corollary we can give} necessary conditions in terms of the dominance order for the existence of variants of standard immaculate tableaux.

\begin{corollary}\label{cor:nec_cond_dom}
Let $\alpha,\beta \vDash n$. 
\begin{enumerate}
    \item If $\SIT_{\mathfrak{S}^*}(\alpha;\beta) \neq \emptyset$, then $\beta \leq_{d} \alpha$.
    \item If $\SRIT_{\mathfrak{S}^*}(\alpha;\beta) \neq \emptyset$, then $\beta^r \leq_{d} \alpha^r$.
    \item If $\SIT_{R\mathfrak{S}^*}(\alpha;\beta) \neq \emptyset$, then $\beta^c \leq_{d} \alpha$.
    \item If $\SRIT_{R\mathfrak{S}^*}(\alpha;\beta) \neq \emptyset$ , then $\beta^t \leq_{d} \alpha^r$.
\end{enumerate}
\end{corollary}

\begin{proof}
Statement 1 is Proposition \ref{prop:nec_cond_L_dom}. Statement 2 follows from 1 by the flip map and Lemma \ref{lem:flip_reverses_content}. Statement 3 follows from 1 since $\SIT_{R\mathfrak{S}^*}(\alpha;\beta) = \SIT_{\mathfrak{S}^*}(\alpha;\beta^c)$. Finally, Statement 4 follows from 3 using the flip map and Lemma \ref{lem:flip_reverses_content}.
\end{proof}

\begin{remark}
    One can derive weaker necessary conditions for $\SIT_{\dI^*}(\alpha ; \beta) \neq \emptyset$ \svw{and the other variants} using $\beta \leq_d \alpha$ implies $\beta \leq_\ell \alpha$ and $\ell(\alpha) \leq \ell(\beta)$.
\end{remark}

\subsubsection{Derived from \svw{the} $R\mathfrak{S}^*$-descent composition}

Recall that a standard immaculate tableau with $\mathfrak{S}^*$-descent composition can be interpreted as a row-strict standard immaculate \svw{tableau,} and the $R\mathfrak{S}^*$-descent composition is the complement of the $\mathfrak{S}^*$-descent composition. As well, by Remark \ref{rem:descent_comp_as_runs}, the $R\mathfrak{S}^*$-descent composition records how the entries of a standard immaculate tableau, from $1$ to $n$, can be read strictly up rows, forming maximal runs of integers. Interpreting this will give another necessary condition on the existence of a \svw{tableau} of shape $\alpha$ and $\mathfrak{S}^*$-descent composition $\beta^c$. For a composition $\alpha = (\alpha_1,\alpha_2,\ldots,\alpha_\ell)$, define $\max \alpha = \max \{ \alpha_i \}_{i=1}^\ell$.

\begin{proposition}\label{prop:nec_cond_L_from_RG_des}
\svw{Let $\alpha,\beta \vDash n$ be} such that $\SIT_{\mathfrak{S}^*}(\alpha;\beta) \neq \emptyset$. Then $\ell(\alpha) \geq \max \beta ^c$.
\end{proposition}

\begin{proof}
Since complement is an involution, we can assume $T \in \SIT_{\mathfrak{S}^*}(\alpha;\beta^c) = \SIT_{R\mathfrak{S}^*}(\alpha;\beta)$. Consider $1 \leq i \leq \ell(\beta)$. Then, $\beta_{1}+\cdots+\beta_{i-1}+1,\svw{\ldots,}\beta_{1}+\svw{\cdots}+\beta_{i}$ are entries in $T$ with row indices $R_{\beta_{1}+\cdots+\beta_{i-1}+1},\ldots,R_{\beta_{1}+\cdots+\beta_{i}}$ respectively, such that the sequence $(R_{m})$ is strictly increasing. Since $R_{\beta_{1}+\cdots+\beta_{i-1}+1} \geq 1$, it follows that $R_{\beta_{1}+\cdots+\beta_{i}} \geq \beta_{i}$. Thus, $\ell(\alpha) \geq \max_{1 \leq i \leq \ell(\beta)} R_{\beta_{1}+\cdots+\beta_{i}} \geq \max_{1 \leq i \leq \ell(\beta)} \beta_{i} = \max \beta$. 
\end{proof}

\begin{example}
    Refer to the standard immaculate tableau $T$ in Example \ref{ex:tab_ex_for_nec_conds} with shape $\alpha = (4,5,1)$ and $\mathfrak{S}^*$-descent composition $\beta = (2,2,2,3,1)$. Note that $\ell(\alpha)=3 \geq \max \beta^c = \max (1,2,2,2,1,2)=2$.
\end{example}

\begin{corollary}\label{cor:nec_cond_L_from_RG_des}
Let $\alpha,\beta \vDash n$. 
\begin{enumerate}
    \item If $\SIT_{\mathfrak{S}^*}(\alpha;\beta) \neq \emptyset$, then $\ell(\alpha) \geq \max \beta ^c$.
    \item If $\SRIT_{\mathfrak{S}^*}(\alpha;\beta) \neq \emptyset$, then $\ell(\alpha)\geq \max \beta^c$. 
    \item If $\SIT_{R\mathfrak{S}^*}(\alpha;\beta) \neq \emptyset$, then $\ell(\alpha)\geq \max\beta$.
    \item If $\SRIT_{R\mathfrak{S}^*}(\alpha;\beta) \neq \emptyset$, then $\ell(\alpha) \geq \max \beta$.
\end{enumerate}
\end{corollary}

\begin{proof}
The proof is identical to the proof of Corollary \ref{cor:nec_cond_dom} using Proposition \ref{prop:nec_cond_L_from_RG_des} for Statement 1.
\end{proof}

\subsection{Sufficient conditions for $\SIT_{\mathfrak{S}^*}(\alpha;\beta) \neq \emptyset$}

As a partial converse to the above necessary condition concerning the dominance order in Proposition \ref{prop:nec_cond_L_dom}, we give two conditions on $\alpha,\beta$ that, along with $\beta\leq_{d} \alpha$, guarantee that $\SIT_{\mathfrak{S}^*}(\alpha;\beta)\neq \emptyset$.

\begin{proposition}\label{prop:suff_cond_equal_len}
    Let $\alpha,\beta \vDash n$. Suppose $\beta \leq_{d} \alpha$ and $\ell(\alpha)=\ell(\beta)$. Then $\SIT_{\mathfrak{S}^*}(\alpha;\beta) \neq \emptyset$.
\end{proposition}

\begin{proof}
Given $\alpha,\beta \vDash n$ such that $\beta \leq_d \alpha$ and $\ell(\alpha) = \ell(\beta)$, 
    we construct a standard immaculate tableau $T$ of shape $\alpha$ and $\mathfrak{S}^*$-descent composition $\beta$. Let $\ell=\ell(\alpha)=\ell(\beta)$.
    
\svw{We will do an induction on the indices of $\beta$.} Since $\beta_{1} \leq \alpha_{1}$, we can assign entries $1,2,\dots,\beta_{1}$ 
to row 1, such that the row is increasing left to right, i.e. to cells $(1,1), (1,2), \ldots, (1,\beta_1)$ respectively.
\svw{Now suppose that} $k<\ell$ and the entries $\{1,2, \ldots, \beta_{1}+\cdots+\beta_{k} \}$ have been assigned among rows $1,2,\ldots, k$. 
Then, let $\beta_{1}+\cdots+\beta_{k}+1 = T(k+1,1)$. Note this is possible since $k+1 \leq \ell$. 
Then assign entries $\{ \beta_{1}+\cdots+\beta_{k}+2,\ldots,\beta_{1}+\cdots+\beta_{k+1} \}$ \svw{in order} among free cells, such that each row is increasing left to right and the row indices are weakly decreasing:
	In particular, if $\alpha_{k+1} \geq 2$, place entries $\{ \beta_1 + \cdots + \beta_k + 2, \ldots, \beta_1 + \cdots + \beta_k + \alpha_{k+1} \}$ into cells $(k+1,2),(k+1,3),\ldots,(k+1,\alpha_{k+1})$ respectively, filling row $k+1$. If $\beta_{k+1} > \alpha_{k+1}$, place as many of the remaining entries as possible into \svwrev{`free'} cells in row $k$, i.e. cells that have not had an entry assigned yet, such that row $k$ is increasing left to right, and so on with the rows below \svw{in order.}
 
	Such a filling is possible since the number of free cells in rows $1,2,\ldots, k+1$ is $\sum_{i=1}^{k+1}\alpha_{i} - \sum_{i=1}^k \beta_{i}$, while the number of entries that have to be assigned in this step is $\beta_{k+1} = \sum_{i=1}^{k+1}\beta_{i} - \sum_{i=1}^k \beta_{i}$. Since $\beta \leq_d \alpha$, this there is at least as many free cells as entries that have to be assigned.
    
    Note, since cells $(1,1),(2,1),\ldots,(k+1,1)$ have already been filled, entries $\{ \beta_{1} + \cdots + \beta_{k}+2,\ldots,\beta_{1} + \cdots + \beta_{k+1} \} $ will be assigned to columns $\geq 2$. 

\svw{We next} check that the constructed tableau is in fact an immaculate tableau. Observe that the rows are increasing left to right by construction. As well, the first column of $T$ is increasing since its content is exactly (reading from bottom to top) $ \{ 1 \} \cup (\Set(\beta)+1)$. Thus, $T \in \SIT(\alpha)$.

\svw{Finally, it} remains to verify $\Des_{\mathfrak{S}^*}(T)=\Set(\beta)$, i.e. an entry $m$ of $T$ is a $\mathfrak{S}^*$-descent of $T$ if and only if $m= \beta_1 + \cdots + \beta_i$ for some $i$. By construction, \svw{for all $i$} the row indices of \svw{$\{\beta_1+\cdots+\beta_{i-1} +1, \ldots,\beta_1+\cdots+\beta_{i} \}$} are weakly decreasing  and so $\svw{\{\beta_1+\cdots+\beta_{i-1} +1, \ldots,\beta_1+\cdots+\beta_{i}-1 \}} \subseteq (\Des_{\mathfrak{S}^*}(T))^c$. As well, the row index of $\beta_1+\cdots+\beta_i$ is \svw{at most} $i$, while the row index of $\beta_1+\cdots+\beta_i+1$ is $i+1$, and so $\beta_1+\cdots+\beta_i \in \Des_{\mathfrak{S}^*}(T)$.
\end{proof}

\begin{example}\label{ex:suff_cond_equal_length_ex}
    A standard immaculate tableau of shape $\alpha=(1,4,3,2,5)$ and $\dI^*$-descent composition $\beta=(1,3,2,1,8)$, as constructed in the proof of Proposition \ref{prop:suff_cond_equal_len}.
    \begin{center}
       \begin{ytableau}
    8 & 9 & 10 & 11 & 12 \\ 
    7 & 13 \\ 
    5 & 6 & 14 \\ 
    2 & 3 & 4 & 15 \\ 
    1 
\end{ytableau}
    \end{center}
\vspace{0.1cm}
\end{example}

\begin{corollary}\label{cor:suff_conds_SIT_equal_len}
Let $\alpha,\beta \vDash n$. 
\begin{enumerate}
    \item If $\beta \leq_d \alpha$ and $\ell(\alpha)=\ell(\beta)$, then $\SIT_{\mathfrak{S}^*}(\alpha;\beta) \neq \emptyset$.
    \item If $\beta^r \leq_d \alpha^r$ and $\ell(\alpha)=\ell(\beta)$, then $\SRIT_{\mathfrak{S}^*}(\alpha;\beta) \neq \emptyset$.
    \item If $\beta^c \leq_d \alpha$ and $\ell(\alpha)=n+1 - \ell(\beta)$, then $\SIT_{R\mathfrak{S}^*}(\alpha;\beta) \neq \emptyset$. 
    \item If $\beta^t \leq_d \alpha^r$ and $\ell(\alpha)=n+1 - \ell(\beta)$, then $\SRIT_{R\mathfrak{S}^*}(\alpha;\beta) \neq \emptyset$.
\end{enumerate}
\end{corollary}

\begin{proof}
    Statement 1 is Proposition \ref{prop:suff_cond_equal_len}. 
    Statement 2 follows from 1 by the flip map and Lemma \ref{lem:flip_reverses_content}. 
    Statement 3 follows from 1 since $\SIT_{R\mathfrak{S}^*}(\alpha;\beta) = \SIT_{\mathfrak{S}^*}(\alpha;\beta^c)$ and by Lemma \ref{lem:len_complement}. Finally, Statement 4 follows by 3 using the flip map and Lemma \ref{lem:flip_reverses_content}. 
\end{proof}

We derive a more specialized sufficient condition for when $\beta$ is a diving-board, which will be useful to us in \svw{Section} \ref{ch:dI_images}.

\begin{proposition}\label{prop:suff_cond_diving_board}
Let $\alpha,\beta \vDash n$ with $\ell=\ell(\alpha)$. Suppose $\alpha$ is not a diving-board and $\beta$ is a diving-board of length $\ell+1$, such that the \svw{long row} occurs at index $\ell$, i.e. $\beta = (1^{\ell-1},n-\ell, 1)$. 
Then $\SIT_{\mathfrak{S}^*}(\alpha,\beta) \neq \emptyset$.
\end{proposition}

\begin{proof}
    Since $\alpha$ is not a diving-board, there \svw{exists} $i<j$ such that $\alpha_{i},\alpha_{j} \geq 2$. Let $i$ be the minimal index such that $\alpha_{i}\geq 2$.
We construct a standard immaculate tableau of shape $\alpha$ and $\mathfrak{S}^*$-descent composition $\beta$. 
Fill the diagram  \svw{of $\alpha$} as follows to get $T \in \SIT(\alpha)$:
\begin{itemize}
    \item Assign $1,2,\dots,\svw{\ell}$ \svw{to} the first column, such that the column is increasing from bottom to top, i.e. $i = T(i,1)$ for $i=1,2,\ldots,\ell$.
    \item Assign $n$ rightmost in row $j$, i.e. $n = T(j,\alpha_j)$.
    \item Assign $n-1$ rightmost in row $i$, i.e. $n = T(i,\alpha_i)$.
    \item Distribute the entries $\ell+1,\ell+2,\dots,n-2$ among the remaining cells, such that the rows are increasing left to right, and the row indices are weakly decreasing, i.e. $R_{\ell+1} \geq R_{\ell+2} \geq \cdots \geq R_{n-2}$. This can be done by placing the entries in order in the free cells from left to right in the rows, going from top to bottom. 
\end{itemize}

By construction, \svw{$T\in \SIT(\alpha)$.} We verify $\Des_{\mathfrak{S}^*}(T)=\beta$ by checking the $\mathfrak{S}^*$-descent set of $T$ is equal to $\Set(\beta)=\{ 1,2,\ldots,\ell-1,n-1 \}$. Observe that for \svw{every} entry $m \leq \ell-1$, both $m,m+1$ are in column 1, so $m \in \Des_{\mathfrak{S}^*}(T)$, since the first column is increasing from bottom to top. As well, $\ell \not\in \Des_{\mathfrak{S}^*}(T)$ since $\ell$ is in the topmost row. For $m \in \{ \ell+1,\dots,n-3 \}$, $m \not\in \Des_{\mathfrak{S}^*}(T)$ since \svw{the} row indices of $m,m+1$ are weakly decreasing by construction. Since $i$ is minimal such that $\alpha_{i} \geq 2$, it must be that $R_{n-2}\geq i=R_{n-1}$ and so $n-2 \not\in \Des_{\mathfrak{S}^*}(T)$. Finally, since $i<j$, $n-1 \in \Des_{\mathfrak{S}^*}(T)$.
\end{proof}

\begin{example}\label{ex:suff_cond_diving_board_ex}
    A standard immaculate tableau of shape $\alpha=(1,1,3,2,4,1)$ and $\dI^*$-descent composition $\beta=(1,1,1,1,1,6,1)$, as constructed in the proof of Proposition \ref{prop:suff_cond_diving_board}. Here, $i=3$ and $j=4$.
    \begin{center}
\begin{ytableau}
    6 \\ 
    5 & 7 & 8 & 9 \\ 
    4 & 12 \\ 
    3 & 10 & 11 \\ 
    2 \\
    1 
\end{ytableau}
    \end{center}
    \vspace{0.1cm}
\end{example}

\begin{corollary}\label{cor:suff_conds_SIT_diving_brd}
Let $\alpha,\beta \vDash n$ with $\ell=\ell(\alpha)$. Suppose $\alpha$ is not a \svw{diving-board.}
\begin{enumerate}
    \item If $\beta = (1^{\ell-1}, n-\ell,1)$, then $\SIT_{\mathfrak{S}^*}(\alpha;\beta) \neq \emptyset$.
    \item If $\beta = (1, n-\ell,1^{\ell-1})$, then $\SRIT_{\mathfrak{S}^*}(\alpha;\beta) \neq \emptyset$.
    \item If $\beta = (\ell, 1^{n-\ell-2},2)$, then $\SIT_{R\mathfrak{S}^*}(\alpha;\beta) \neq \emptyset$. 
    \item If $\beta = (2, 1^{n-\ell-2},\ell)$, then $\SRIT_{R\mathfrak{S}^*}(\alpha;\beta) \neq \emptyset$.
\end{enumerate}
\end{corollary}

\begin{proof}
    The proof is identical to the proof of Corollary \ref{cor:suff_conds_SIT_equal_len} using Proposition \ref{prop:suff_cond_diving_board} for Statement 1.
\end{proof}

\section{Classification of dual immaculate functions under automorphisms}\label{ch:dI_images}

Consider the expansion of a variant of a dual immaculate function in the dual immaculate basis, 
\[
f(\mathfrak{S}^*_{\alpha}) = \sum_{\beta \vDash n} d_{\beta}\svw{\mathfrak{S}^*_{\beta}}
\]
for $f \in \{ \rho, \psi, \omega \}$.

In this section, we \svw{give and prove our} main theorem which classifies the compositions $\alpha$ such that this expansion is `simplest', i.e. the expansion consists of exactly one term 
\[
f(\mathfrak{S}^*_{\alpha})=\mathfrak{S}^*_{\beta^f_{\alpha}},
\]
where $\beta^f_{\alpha}$ is the leading term \svw{given in} Section \ref{ch:leading_terms}.

\begin{theorem}\label{thm:main}
Let $\alpha \vDash n$ \svw{and} $f \in \{  \rho, \psi \}$. Then, $f(\mathfrak{S}^*_{\alpha})$ has exactly one term when expanded in the dual immaculate basis if and only if $\alpha$ is bottom-aligned hook. In this case, 
\begin{align*}
\rho(\mathfrak{S}^*_{\alpha}) &= \mathfrak{S}^*_{\alpha} ,\\
\psi(\mathfrak{S}^*_{\alpha}) &= \mathfrak{S}^*_{\alpha^t}.
\end{align*}
Furthermore, $\omega(\mathfrak{S}^*_{\alpha})$ has exactly one term when expressed in the dual immaculate basis if and only if $\alpha$ is a bottom- or top-aligned hook. In these cases, \[
\omega(\mathfrak{S}^*_{\alpha})=\mathfrak{S}^*_{\alpha^t}.
\]
\end{theorem}

\begin{proof} \svw{We first prove the sufficient direction, and then the more intricate necessary direction that will require  case analyses.}

\

\svw{\emph{Proof of the sufficient direction.}} Suppose $\alpha$ is a bottom-aligned hook. By Proposition \ref{prop:dI_symmetric}, $\mathfrak{S}^*_{\alpha}$ is symmetric with $\mathfrak{S}^*_{\alpha}=s_{\alpha}$. As well, $\rho$ restricts to the identity morphism on $\Sym$, so 
\[
\rho(\mathfrak{S}^*_{\alpha}) = \rho(s_\alpha)=s_{\alpha}=\mathfrak{S}^*_{\alpha}.
\]
On the other hand, $\psi,\omega$ restrict to $\omega : \Sym \to \Sym$, with $\omega \ME{(s_{\alpha}) = s_{\alpha^t}}$ when $\alpha$ is a bottom-aligned hook by the comment just before Section \ref{ch:dI_bg}. By Proposition \ref{prop:transpose_preserves_hooks}, since $\alpha$ is a bottom-aligned hook, so is $\alpha^t$. Thus, for $f \in \{\psi, \omega \}$
\[
f(\mathfrak{S}^*_{\alpha}) = f(s_\alpha)=s_{\alpha^t}=\mathfrak{S}^*_{\alpha^t}.
\]
The result for bottom-aligned hooks follows. 

Suppose $\alpha$ is a top-aligned hook. By Proposition \ref{prop:ex_1_descent}, $\mathfrak{S}^*_{\alpha} = F_\alpha$ and so $\omega(\mathfrak{S}^*_{\alpha})=\omega(F_{\alpha})=F_{\alpha^t}$ by Definition \ref{def:qsym_involutions}. Since $\alpha^t$ is a top-aligned hook by \svw{Proposition} \ref{prop:transpose_preserves_hooks}, we see that $F_{\alpha^t}=\mathfrak{S}^*_{\alpha^t}$, again by Proposition \ref{prop:ex_1_descent}. Thus,
\[
\omega(\mathfrak{S}^*_{\alpha}) = \omega(F_\alpha)=F_{\alpha^t}=\mathfrak{S}^*_{\alpha^t}.
\]
The result for top-aligned hooks follows, \svw{and this concludes the proof of the sufficient direction.}

\

\svw{\emph{Proof of the necessary direction.}}
Assume $\alpha \vDash n$ has length $\ell$ and is not a bottom-aligned hook. 
We show $f(\mathfrak{S}^*_\alpha) \neq \mathfrak{S}^*_{\beta^f_\alpha}$ and hence $f(\mathfrak{S}^*_\alpha)$ is not exactly one term when expanded in the dual immaculate basis. \\

\svw{\textsc{\underline{Case $f= \rho$.}}} We begin with $f= \rho$ and consider two cases for $\alpha$. \\

Since $\beta^\rho_{\alpha}=(n-\ell+1,1^{\ell-1})$ is a bottom-aligned hook by Proposition \ref{prop:rho_leading_term}, by our earlier explicit computation in Proposition \ref{prop:diving_board_explicit_comp},
\[
\mathfrak{S}^*_{\beta^\rho_{\alpha}} = \sum_{S} \svw{F_{\Comp([n-1] \setminus (S-1))}}
\]
where the sum is over all $S \subset \{ 2,\dots,n \}$ with $|S| = n-\ell$. Note that the number of terms in the \svw{fundamental} expansion of $\mathfrak{S}^*_{\beta^\rho_{\alpha}}$ is ${n-1 \choose n-\ell}$. \\

\textit{Case 1.} Suppose $\alpha$ is a \svw{diving-board.}

Observe that since $\alpha$ is not a bottom-aligned hook, the \svw{long row} of $\alpha$ occurs at index $i>1$. The number of terms in the fundamental expansion of $\rho(\mathfrak{S}^*_{\alpha})$ is equal to the number of terms in the fundamental expansion of $\mathfrak{S}^*_{\alpha}$ since $\rho$ is an involution \svw{such that $\rho(F_\alpha)=F_{\alpha^r}$.} Again, using the explicit computation in Proposition \ref{prop:diving_board_explicit_comp}, the number of terms in the fundamental expansion of $\mathfrak{S}^*_{\alpha}$ is ${n-i \choose n-\ell} < {n-1 \choose n-\ell}$ since $\ell(\alpha) = \ell$ by assumption. Thus, $\rho(\mathfrak{S}^*_{\alpha})\neq \mathfrak{S}^*_{\beta^\rho_{\alpha}}$.\\ 

\textit{Case 2.} Suppose $\alpha$ is not a \svw{diving-board.} 

In this case, we identify a \svw{diving-board} $\gamma$ such that $[F_{\gamma}] \rho (\mathfrak{S}^*_{\alpha}) > 0$ while $[F_{\gamma}]\mathfrak{S}^*_{\beta^\rho_{\alpha}} = 0$. Let 
\[
\gamma=(1, n - \ell , 1^{\ell-1}).
\]
Since $\alpha$ is not a diving-board, $n - \ell \geq 2$ and $\ell \geq 2$. Thus, $\ell(\gamma)=\ell+1$.

By Propositions \ref{prop:dI_over_F} and \ref{prop:variants_dI_as_images}, $|\SRIT_{\mathfrak{S}^*}(\alpha^r;\gamma)|=[F_{\gamma}]\rho(\mathfrak{S}^*_\alpha)$. Thus, the sufficient condition in Corollary \ref{cor:suff_conds_SIT_diving_brd} implies $[F_{\gamma}]\rho(\mathfrak{S}^*_{\alpha})>0$.

Since $\beta^\rho_{\alpha}$ is a \svw{diving-board,} to show that $|\SIT_{\mathfrak{S}^*}(\beta^\rho_{\alpha};\gamma)| = [F_{\gamma}]\mathfrak{S}^*_{\beta^\rho_{\alpha}}=0$, we use the fundamental expansion of $\mathfrak{S}^*_{\beta^\rho_{\alpha}}$  in Proposition \ref{prop:diving_board_explicit_comp}. Every term in the fundamental expansion of $\mathfrak{S}^*_{\beta^\rho_{\alpha}}$ is indexed by a composition $\Comp([n-1] \setminus(S-1))$, where $|S|=n-\ell$. In particular, every term in the fundamental expansion of $\mathfrak{S}^*_{\beta^\rho_{\alpha}}$ is indexed by a composition of length $\ell$. However, $\ell(\gamma)=\ell+1$ and so $[F_{\gamma}]\mathfrak{S}^*_{\beta^\rho_{\alpha}}=0$. \svw{Thus, $\rho(\mathfrak{S}^*_\alpha) \neq \mathfrak{S}^*_{\beta^\rho_\alpha}$.}\\

\svw{\textsc{\underline{Case $f= \psi$.}}} We now look at $f=\psi$ and again consider two cases for $\alpha$.\\

\textit{Case 1.} Suppose $\alpha$ is not a partition. 

In this case, we identify a composition $\gamma$ such that $[F_{\gamma}] \psi (\mathfrak{S}^*_{\alpha}) = 0$ while $[F_{\gamma}]\mathfrak{S}^*_{\beta^\psi_\alpha} > 0$. Since $\alpha$ is not a partition, let $i$ be the minimal index such that $\alpha_{i} > \alpha_{i-1}$ and let
\[
\gamma^c = (\alpha_{1},\ldots,\alpha_{i-2},\alpha_{i-1}+1,\alpha_{i}-1,\alpha_{i+1},\ldots, \alpha_\ell).
\]
Note that \svw{$\gamma^c \not\leq_{d} \alpha$.}

By Propositions \ref{prop:dI_over_F} and \ref{prop:variants_dI_as_images}, $|\SIT_{R\mathfrak{S}^*}(\alpha;\gamma)|=[F_{\gamma}]\psi(\mathfrak{S}^*_\alpha)$. Thus, the necessary condition in Corollary \ref{cor:nec_cond_dom} implies $[F_{\gamma}]\psi(\mathfrak{S}^*_{\alpha})=0$. 

To show that $|\SIT_{\mathfrak{S}^*}(\beta^\psi_\alpha;\gamma)| = [F_{\gamma}]\mathfrak{S}^*_{\beta^\psi_\alpha} > 0$, we use the sufficient condition in Proposition \ref{prop:suff_cond_equal_len}. We need to show that $\gamma \leq_{d} \beta^\psi_\alpha$ and $\ell(\gamma) = \ell(\beta^\psi_\alpha)$.

To show $\gamma \leq_{d} \beta^\psi_\alpha$, we use Proposition \ref{prop:dom_complement} and instead prove the equivalent statement $\gamma^c \geq_{d} (\beta^\psi_\alpha)^c$. 
Let $\beta^\psi_\alpha= (c_{1},c_{2},\dots)$ and note that $c_{1} = \ell(\alpha)$ by Proposition \ref{prop:psi_leading_term}. If $\ell(\alpha)=1$ then $\alpha$ is a partition and so $c_{1} = \ell(\alpha) \geq 2$. As well, if $\ell(\beta^\psi_\alpha)=1$ then $\alpha=(1^n)$ is a partition by Proposition \ref{prop:psi_leading_term}. Thus, $\ell(\beta^\psi_\alpha)\geq 2$. Thus, we can write
\[
(\beta^\psi_\alpha)^c=(1^{c_{1}-1},a,\dots)
\]
for some $a \geq 1$. 

For $k<c_{1}$, $\sum_{i=1}^k ((\beta^\psi_\alpha)^c)_{i} = k$ while $\sum_{i=1}^k \gamma^c_{i} \geq \sum_{i=1}^k \alpha_{i} \geq k$ since each part $\alpha_{i} \geq 1$. 

For $k = c_1$, $\sum_{i=1}^k \gamma^c_{i} = n \geq \sum_{i=1}^k ((\beta^\psi_\alpha)^c)_{i}$ since $\beta^\psi_\alpha$ is a composition of $n$ and $c_{1} = \ell(\alpha) = \ell(\gamma^c)$.
Thus, $\gamma^c \geq_{d} (\beta^\psi_\alpha)^c$ as required. \\

Since $\gamma \leq_d \beta^\psi_\alpha$, it follows that $\ell(\beta^\psi_\alpha) \leq \ell(\gamma)$. Thus, it remains to show $\ell(\gamma) \leq \ell(\beta^\psi_\alpha)$.
By \svw{Proposition \ref{prop:psi_leading_term} describing $\beta^\psi_\alpha$,} the length of $\beta^\psi_\alpha$ is the number of columns in the diagram of $\alpha$, given by $\max \alpha$. Suppose $i = \arg \max \alpha$. 
By Lemma \ref{lem:len_complement}, $\ell(\gamma) = n-(\ell(\gamma^c)-1)=n-(\ell(\alpha)-1)$. Thus, we can think of $\ell(\gamma)$ as the number of cells in the diagram of $\alpha$, excluding the cells with indices $(1,1),(2,1),\dots, (i-1,1),(i+1,1), \ldots ,(\ell(\alpha),1)$. In particular, this subset of cells contains the $i$-th row of $\alpha$, and so $\ell(\gamma)\geq \max\alpha = \ell(\beta^\psi_\alpha)$. \svw{Thus, $\psi(\mathfrak{S}^*_\alpha) \neq \mathfrak{S}^*_{\beta^\psi_\alpha}$.}\\

\textit{Case 2.} Suppose $\alpha$ is a partition. 

In this case, we identify a composition $\gamma$ such that $[F_{\gamma}] \psi (\mathfrak{S}^*_{\alpha}) > 0$ while $[F_{\gamma}]\mathfrak{S}^*_{\beta^\psi_\alpha} = 0$. Let 
\[
\gamma^c = (1,\alpha_{1}-1+\alpha_{2},\alpha_{3},\alpha_{4},\ldots, \alpha_\ell).\]
Note that $\gamma^c \leq_{d} \alpha$ and $\ell(\gamma^c)=\ell(\alpha)$.

By Propositions \ref{prop:dI_over_F} and \ref{prop:variants_dI_as_images}, $|\SIT_{R\mathfrak{S}^*}(\alpha;\gamma)|=[F_{\gamma}]\psi(\mathfrak{S}^*_\alpha)$. Thus, the sufficient condition in Corollary \ref{cor:suff_conds_SIT_equal_len} and Lemma \ref{lem:len_complement} imply $[F_{\gamma}]\psi(\mathfrak{S}^*_{\alpha})>0$. 

To show $|\SIT_{\mathfrak{S}^*}(\beta^\psi_\alpha;\gamma)| = [F_{\gamma}]\mathfrak{S}^*_{\beta^\psi_\alpha} = 0$, we use the necessary condition in Proposition \ref{prop:nec_cond_L_from_RG_des}. We need to show $\ell(\beta^\psi_\alpha)< \max \gamma^c$. 
By Proposition \ref{prop:psi_leading_term}, the length of $\beta^\psi_\alpha$ is the number of columns in \svw{the} diagram \svw{of} $\alpha$, given by $\max\alpha$. Since $\alpha$ is a partition, $\ell(\beta^\psi_\alpha)=\alpha_{1}$. Since $\alpha$ \svw{is not} a bottom-aligned hook \svw{by assumption,} we have that $\alpha_{2} \geq 2$. Thus, 
\[
\max \gamma^c \geq \alpha_{1}-1+\alpha_{2} \geq \alpha_{1}-1+2 > \alpha_{1} = \ell(\beta^\psi_\alpha),
\]
and we are done. \svw{Thus, $\psi(\mathfrak{S}^*_\alpha) \neq \mathfrak{S}^*_{\beta^\psi_\alpha}$.}\\

\svw{\textsc{\underline{Case $f= \omega$.}}} Finally, we look at $f=\omega$. Assume $\alpha \vDash n$ and $\alpha$ is not a bottom- or \svw{top-aligned} hook. 
We consider two cases for $\alpha$ and in both of the following cases, we identify $\gamma$ such that $[F_{\gamma}] \omega (\mathfrak{S}^*_{\alpha}) = 0$ while $[F_{\gamma}]\mathfrak{S}^*_{\beta^\omega_\alpha} > 0$. \\

\textit{Case 1.} Suppose $\alpha$ is not a \svw{diving-board.}

Let $\gamma$ be the top-aligned hook of size $n$ with $\ell(\gamma)=\ell(\beta^\omega_\alpha)$, i.e. 
\[
\gamma = (1^{\ell(\beta^\omega_\alpha)-1}, n-\ell(\beta^\omega_\alpha)+1).
\]
By Lemma \ref{lemma:characterize_alpha_divboard_using_beta}, since $\alpha$ is not a diving-board, $\ell(\alpha)+\ell(\beta^\omega_\alpha)\leq n$. Thus, $\svw{\max \gamma =} \gamma_{\ell(\gamma)} = n - \ell(\beta^\omega_\alpha) + 1 \geq \ell(\alpha)+1 > \ell(\alpha) \svw{=\ell(\alpha ^r)}$. 

By Propositions \ref{prop:dI_over_F} and \ref{prop:variants_dI_as_images}, $|\SRIT_{R\mathfrak{S}^*}(\alpha^r;\gamma)|=[F_{\gamma}]\omega(\mathfrak{S}^*_\alpha)$. Thus, the necessary condition in Corollary \ref{cor:nec_cond_L_from_RG_des} implies $[F_{\gamma}]\omega(\mathfrak{S}^*_{\alpha})=0$. 

To show $|\SIT_{\mathfrak{S}^*}(\beta^\omega_\alpha; \gamma)| = [F_\gamma]\mathfrak{S}^*_{\beta^\omega_\alpha} > 0$, we use the sufficient condition in Proposition \ref{prop:suff_cond_equal_len}. We need to show $\ell(\gamma) = \ell(\beta^\omega_\alpha)$ and $\gamma \leq_d \beta^\omega_\alpha$. 
Since $\gamma$ is a top-aligned hook, $\gamma \leq_{d} \beta^\omega_\alpha$ and $\ell(\gamma) = \ell(\beta^\omega_\alpha)$ by definition of $\gamma$. \svw{Thus, $\omega(\mathfrak{S}^*_\alpha) \neq \mathfrak{S}^*_{\beta^\omega_\alpha}$.}
\\

\textit{Case 2.} Suppose $\alpha$ is a \svw{diving-board.}

Since $\alpha$ is not a top- or bottom-aligned hook, suppose the \svw{long row} of length $k$ occurs at index $1 < i < \ell(\alpha)$. Note $k \geq 2$, otherwise $\alpha = (1^n)$ is a bottom-aligned hook. As well, $\ell(\alpha)\geq 3$, otherwise $\alpha$ is a top- or \svw{bottom-aligned} hook. 
By Lemma \ref{lemma:characterize_alpha_divboard_using_beta}, $n+1 = \ell(\alpha)+\ell(\beta^\omega_\alpha)$. Using this and $\ell(\alpha) \geq 3$, we have that $n - \ell(\beta^\omega_\alpha) = \ell(\alpha) - 1 \geq 2$. Thus, let $\gamma$ be the \svw{diving-board} of size $n$ such that $\ell(\gamma)=\ell(\beta^\omega_\alpha)+1$ and $\gamma_{\ell(\beta^\omega_\alpha)} \geq 2$. In particular, 
\[
\gamma = (1^{\ell(\beta^\omega_\alpha)-1}, n-\ell(\beta^\omega_\alpha), 1).
\]

By Propositions \ref{prop:dI_over_F} and \ref{prop:variants_dI_as_images}, $|\SRIT_{R\mathfrak{S}^*}(\alpha^r;\gamma)|=[F_{\gamma}]\omega(\mathfrak{S}^*_\alpha)$. By the necessary condition in Corollary \ref{cor:nec_cond_dom} we need to show $\svw{\gamma^t \not\leq_d \alpha}$ to conclude that $[F_{\gamma}]\omega(\mathfrak{S}^*_{\alpha})=0$. Taking the complement and using Proposition \ref{prop:dom_complement} we can equivalently check that $\svw{\gamma^r \not\geq_d \alpha^c}$.

Since $\alpha = (1^{i-1},k,1^{\ell-i})$ with $i > 1$, using the dots and bars diagram of $\alpha$, we see that $\alpha^c_{1}=i>1$. Thus, since $\gamma^r_{1}=\gamma_{\ell(\gamma)}=1<\alpha^c_{1}$, we have that $\svw{\gamma^r \not\geq_d \alpha^c}$.

To show $|\SIT_{\mathfrak{S}^*}(\beta^\omega_\alpha; \gamma)| = [F_\gamma]\mathfrak{S}^*_{\beta^\omega_\alpha} > 0$, we use the sufficient condition in Proposition \ref{prop:suff_cond_diving_board}. It remains to check that $\beta^\omega_\alpha$ is not a diving-board. Since $\alpha$ is not a top- or bottom-aligned hook this follows by Lemma \ref{lemma:explicit_beta_if_alpha_divboard}, \svw{and we are done. Thus, $\omega(\mathfrak{S}^*_\alpha) \neq \mathfrak{S}^*_{\beta^\omega_\alpha}$.}
\end{proof}

\section*{Acknowledgements}
\svw{The authors would like to thank Spencer Daugherty, Jinting Liang and Mitchell Ryan for helpful conversations.}

\end{document}